\documentclass[a4paper,11pt]{amsart}

\usepackage[utf8]{inputenc}
\usepackage{color}
\usepackage[english]{babel}
\usepackage{amsmath,amsthm}
\usepackage{amsfonts}
\usepackage{mathabx}
\usepackage{graphicx}
\newtheorem{thm}{Theorem}
\newtheorem{cor}{Corollary}
\newtheorem{lem}{Lemma}[section]

\newtheorem{prop}[lem]{Proposition}
\usepackage{varioref}
\theoremstyle{definition}
\newtheorem{ex}[lem]{Example}

\newtheorem{defn}[lem]{Definition}
\newtheorem{rem}[lem]{Remark}
\numberwithin{equation}{section}

\title[The Ball-Box Theorem for Non-differentiable Tangent Subbundles ]{The Ball-Box Theorem for a Class of Non-differentiable Tangent Subbundles }

\author{S\.ina T\"urel\.i}

\begin{document}

\begin{abstract}{We show that an analogue of the Ball-Box Theorem holds true for a class of corank 1, non-differentiable tangent subbundles that satisfy a geometric condition. In the final section
of the paper we give examples of such bundles and an application to dynamical systems.}
\end{abstract}

\maketitle

\section{Introduction}
Sub-Riemannian geometry is a generalization of Riemannian geometry which is motivated by very physical and
concrete problems. It is the language for formalizing questions like: ``Can we connect two thermodynamical
states by adiabatic paths?"\cite{Car09}, ``Can a robot with a certain set of movement rules reach everywhere
in a factory?"\cite{Agr04}, ``Can a business man evade tax by following the rules that were set to avoid tax
evasion?", ``By adjusting the current we give to a neural system, can we change the initial phase
of the system to any other phase we want?"\cite{Li10}. However one drawback of current sub-Riemannian
geometry literature is that it almost exclusively focuses on the study of ``smooth systems, which is
sometimes too much to ask for a mathematical subject that has close connections with physical sciences.
For instance, one place where non-differentiable objects appear in a physically motivated mathematical
branch (and which is the main motivation of the authors) is the area of dynamical systems. More specifically
in (partially and uniformly) hyperbolic dynamics, bundles that are only H\"older continuous are quite
abundant and their sub-Riemannian properties (i.e. accessibility and integrability) play an important
role in the description and classification of the dynamics. The aim of this paper is to give a little nudge
to sub-Riemannian geometry in the direction of non-differentiable objects.

To get into more technical details we need some definitions. Let $\Delta$ be a $C^r$ tangent subbundle defined
on a smooth manifold $M$ and $g$ a metric on $\Delta$ (the triple $(M,\Delta,g)$ is called a $C^r$ sub-Riemannian manifold). We will always assume that $\Delta$ is corank $1$ and $dim(M)=n+1$ with $n \geq 2$. A
piecewise $C^1$ path $\gamma$ is said to be \textbf{\emph{admissible}} if it is a.e everywhere tangent to
$\Delta$ (i.e. $\dot{\gamma}(t) \in \Delta_{\gamma(t)}$ for $t$ a.e). We let
$\mathcal{C}_{pq}$ denote the set of length parameterized (i.e. $g(\dot{\gamma}(t),\dot{\gamma}(t))=1$ for $t$ a.e) admissible paths between $p$ and $q$. If
$\mathcal{C}_{pq} \neq \emptyset$ for all $p,q \in U\subset M$ then $\Delta$ is said to be
\textbf{\emph{controllable}} or  \textbf{\emph{accessible}} on $U$.
For smooth bundles, the \textbf{\emph{Chow-Rashevskii Theorem}} says that if $\Delta$ is everywhere
completely non-integrable (i.e. if the smallest Lie algebra $Lie(\Gamma(\Delta))$ generated by smooth sections of $\Delta$ is the whole tangent space at every point), then it is accessible \cite{Agr04}. In particular if
$\Gamma(\Delta)(p) + [\Gamma(\Delta),\Gamma(\Delta)](p) =T_pM$ then such a bundle is called called step 2, completely non-integrable at $p$ \footnote{The accessibility theorem for corank $1$, step 2, completely non-integrable differentiable bundles was actually already
formulated in 1909 by Carath$\acute{\text{e}}$odory with the aim of studying adiabatic paths in thermodynamical
systems \cite{Car09}).}. These are the $C^1$ analogues of the non-differentiable bundles that we will be interested in this paper.
We denote
$$
d_{\Delta}(p,q) = \inf_{\gamma \in\mathcal{C}_{pq}}\ell(\gamma),
$$
where $\ell(\cdot)$ denotes length with respect to the given metric $g$ on $\Delta$. $d_{\Delta}$ is called the
\textbf{\emph{sub-Riemannian metric}}. We let $B_{\Delta}(p,\epsilon)$ denote the ball of radius $\epsilon$
around $p$ with respect to the metric $d_{\Delta}$. Given $p \in M$ and a coordinate neighbourhood $U$,
coordinates $z=(x^1,\dots,x^n,y)$ are called \textbf{\emph{adapted}} for $\Delta$ if,
$p=0$, $\Delta_0 = \text{span}<\frac{\partial}{\partial x^i}|_0>_{i=1}^n$ and $\Delta_q$ is transverse to
$\mathcal{Y}_q = \text{span}<\frac{\partial}{\partial y}|_q>$ for every $q \in U$. Given such coordinates and $\alpha>0$, we define the coordinate weighted box as: $$
B_{\alpha}(0,\epsilon)=\{z=(x,y) \in U \quad | \quad |x^i| \leq \epsilon, \quad |y| \leq \epsilon^{\alpha}\},
$$
where $|\cdot|$ denotes the Euclidean norm with respect to the given adapted coordinates. A specialization of the fundamental the \textbf{\emph {Ball-Box Theorem}} \cite{Bel96,Mon02}, says that if $\Delta$ is a smooth, step 2, completely non-integrable bundle at a point $p$, then, given any smooth adapted coordinate system defined on some small enough neighborhood of $p$, there exists constants $K_1, K_2, \epsilon_0 >0$ such that for all $\epsilon < \epsilon_0$,
$$
B_2(0,K_1\epsilon) \subset B_{\Delta}(0,\epsilon) \subset B_2(0, K_2\epsilon).
$$
Note that here the lower inclusion $B_2(0,K_1\epsilon) \subset B_{\Delta}(0,\epsilon)$ implies accessibility around $p$. This specialized case is also known to hold true for $C^1$ bundles, see \cite{Gro96}.

There are several works ( \cite{Kar11}, \cite{MonMor12}, \cite{Sim10}) that try to generalize the Ball-Box Theorem to the setting of less regular bundles. Each has their own set of assumptions about the regularity or geometric properties of the bundle.

In \cite{Sim10}, they generalize the Ball-Box Theorem to H\"older
continuous bundles following a proof for $C^1$ bundles given in \cite{Gro96}.  The extra assumption is that the bundle $\Delta$ is accessible to start with.
Under this assumption, they prove that for the case of accessible, $\theta-$H\"older, codimension 1 bundle, the inclusion ``$B_{\Delta}(0,\epsilon) \subset B_{\alpha}(0, K_2\epsilon)$'' holds true with $\alpha=1+\theta$ (this inclusion translates as a certain ``lower bound" in the way they choose to express his results in \cite{Sim10}). Although this result does not assume any regularity beyond being  H\"older (which is the weakest regularity assumption in the works we compare), what is lacking is that there is no criterion for accessibility and without the lower inclusion one has no qualitative information about the shape or the volume of the sub-Riemannian ball.

In \cite{MonMor12} they prove the full Ball-Box Theorem under certain Lipschitz continuity assumptions for commutators of the vector-fields involved. In particular working with a collection of vector-fields $\{X_i\}_{i=1}^n$, they say that these vector-fields are completely non-integrable of step $s$ if their Lie derivatives $X_I=[X_{i_1},[\dots,[X_{i_{s-1}},X_{s}]]\cdots]$ up to $s$ iterations are defined and Lipschitz continuous and these $\{X_I\}_{I \in \mathcal{I}}$ span the whole tangent space at a point. And then under these assumptions (and some more less significant technical assumptions) they obtain the usual Ball-Box Theorem for step $s$, completely non-integrable collection of vector-fields.

In \cite{Kar11}, they consider $C^{1+\alpha}$ bundles with $\alpha>0$. In this case the bundle itself is differentiable and therefore the Lie derivatives $[X_i,X_j]$ are defined although only H\"older continuous. Therefore many of the tools of classical theory such as Baker-Hausdorff-Campbell formula are not applicable and extension of the theory already becomes non-trivial. Although the case of step 2, completely non-integrable $C^1$ bundles is already dealt with in \cite{Gro96}, the cited paper considers the general case and does not only study the Ball-Box theorem but several other important theorems from sub-Riemannian geometry.

This paper makes progress toward extending the Ball-Box Theorem to continuous bundles. We establish an analogue of the Ball-Box Theorem (of the step 2 case) and therefore Chow-Rashevskii Theorem for a certain class of non-integrable, continuous, corank 1 bundles that satisfy a geometric condition (explained in the next
section). In particular studying continuous bundles allows us to analyze what are the
important features for giving volume and shape to a sub-Riemannian ball. In the end we can conclude that
certain geometric features are sufficient for obtaining lower bounds on
the volume while regularity also plays an important role for the shape.  The authors believe that the methods presented in this paper can become useful for answering these questions in more generality and this is discussed in section \ref{sec-generalizations}.

After the proof of the main theorems, we give examples of bundles that satisfy these geometric properties and
yet are not differentiable (nor H\"older), we present an application to dynamical systems and we also study some interesting properties of this class of
bundles and pose some questions related to possible generalizations including
measurable bundles (measurable in terms of space variables, not just the time variables which is already
completely covered by classical control theory). After the examples we also compare our results with the other results discussed. But we can before hand say that all the results are somehow related to
each other but are not completely covered by neither (see in particular the discussion following Proposition \ref{prop-example}).

\medskip\noindent
\emph{Acknowledgments:} The author is greatly thankful to the anonymous referee for a lot of improvements, in particular several corrections and useful remarks regarding the distinction between step 2,
completely non-integrable and other cases in corank 1 case. The author was supported by ERC AdG grant no: 339523 RGDD. All the figures were created using \emph{Apache OpenOffice Draw}.

\subsection{Statement of the Theorems}

We assume that $\Delta$ is a corank 1, continuous tangent subbundle defined on a $n+1$ dimensional
smooth manifold $M$ with a given metric $g$ on $\Delta$. Except for some certain general definitions, we will carry out the proof in coordinate neighborhoods. The domain of the coordinate will be possibly a smaller Euclidean box where we work and all the supremums and infumums of the functions defined on this coordinate neighbourhood will be over this domain. Henceforth we denote the domain of any chosen coordinate system by $U$. By $|\cdot|$, we denote the Euclidean norm given by the coordinates and we identify the tangent spaces with $\mathbb{R}^{n+1}$. We denote by $\mathcal{A}^n_0(\Delta)(U)$ the space of continuous differential $n$-forms over $U$
that annihilate $\Delta$, which is seen as a module over $C(U)$. We let $\Omega^n_{r}(U)$ denote the space of $C^r$
differential n-forms over $U$, again as a module over $C^r(U)$. With this notation $\mathcal{A}^n_0(\Delta)(U)$ is a
submodule of $\Omega^n_{0}(U)$. Given a submodule $\mathcal{E}\subset \Omega^n_{r}(U)$, a local basis for this submodule on $U$ is a
finite collection of elements from $\mathcal{E}$, which are linearly independent over $C^r(U)$ and which span $\mathcal{E}$ over $C^r(U)$.  We use the induced norm on these spaces coming from the Euclidean norm to endow them with a Banach space structure. We use $|\cdot|_{\infty}$ and $|\cdot|_{\inf}$ denote the supremum and infimum of the norms of an object over the domain $U$ we are working with. More precisely if $\alpha$ is some $k-$differential form, $|\alpha|_{\infty}=\sup_{q \in U}|\alpha_q|$ and $|\alpha|_{\inf}=\inf_{q \in U}|\alpha_q|$ where generally a subscripted point denotes evaluation at that point. If $Y_1,\dots,Y_k$ are vector-fields then $\alpha(Y_1,\dots,Y_k)$ denotes the function obtained by contracting $\alpha$ with the given vector-fields. Therefore
 $|\alpha(Y_1,\dots,Y_k)|_{\infty}$ and  $|\alpha(Y_1,\dots,Y_k)|_{\inf}$ denotes the supremum and infimum over $U$ of the absolute value of this function. Occasionally when the need arises, we might make the distinction of evaluation points by the notation $\alpha_q(Y_1(q),\dots,Y_k(q))$ or even by  $\alpha_q(Y_1(p),\dots,Y_k(p))$ when we work in coordinates and identify tangent spaces with $\mathbb{R}^{n+1}$.
Finally given a sub-bundle $\Delta$ of $TU$, we denote
$$
|\alpha|_{\Delta}|_{\infty} = \sup_{q \in U, \ v_i \in \Delta_q, \ |v_1 \wedge \dots \wedge v_k|=1} |\alpha_q(v_1 \wedge \dots \wedge v_k)|,
$$
$$
m(\alpha|_{\Delta})_{\inf} = \inf_{q \in U, \ v_i \in \Delta_q, \ |v_1 \wedge \dots \wedge v_k|=1} |\alpha_q(v_1 \wedge \dots \wedge v_k)|.
$$
The first expression is the supremum over $U$ of the norms of $\alpha_q$ seen as linear maps acting on $\bigwedge^k (\Delta_q)$ while the second is the  infimum over $U$ of the conorms of $\alpha_q$.
On the passing we note that given any $v_1,\dots,v_k \in \Delta_q$,
$$
m(\alpha|_{\Delta})_{\inf} \leq  \frac{|\alpha_q(v_1,\dots,v_k)|}{|v_1 \wedge \dots \wedge v_k|} \leq |\alpha|_{\Delta}|_{\infty}.
$$

The next definition is the fundamental geometric regularity property that we require of our non-differentiable bundles in order to endow them with other
nice geometric and analytic properties:

\begin{defn}\label{defn-continuousexteriorderivative} A continuous differential $k$-form $\eta$ is said to
have a continuous exterior differential if there exists a continuous differential $k$+1-form $\beta$ such
that for every $k$-cycle $Y$ and $k$+1 chain $H$ bounded by it, one has that
$$
\int_Y \eta = \int_H \beta.
$$
If such a $\beta$ exists, we suggestively denote it as $d\eta$. A rank $k$ subbundle $\mathcal{E} \subset
\Omega^n_{0}(M)$ is said to be equipped with continuous exterior differential
at $p \in M$, if there exists a neighbourhood $V$ of $p$ on which $\{\eta_i\}_{i=1}^k$ is a local basis of sections of
$\mathcal{E}$ on $V$ and $\{d\eta_i\}_{i=1}^k$ are their continuous exterior differentials. We will occasionally refer to above property as Stokes property and also denote this triple by
$\{V,\eta_i,d\eta_i\}_{i=1}^k$.
\end{defn}

We denote by $\Omega^k_{d}(M)$ the space of all differential $k$-forms that have continuous exterior
differentials. Obviously $\Omega^k_{0}(M) \supset \Omega^k_{d}(M) \supset \Omega^k_{1}(M)\supset
\Omega^k_{2}(M)\supset \dots \ $. It will be the purpose of section \ref{sec-continuousexteriordifferential} to give
non-trivial examples (i.e. non-differentiable and non-H\"older) of such differential forms for $k=1$ and discuss their properties
to illustrate their geometric significance. However, our main theorems only deal with
sub-Riemannian properties of tangent subbundles defined by such differential 1-forms.

\begin{defn}\label{defn-integrable}Let $\Delta$ be a corank $k$, continuous, tangent subbundle. Assume $\mathcal{A}^1_0(\Delta)(M)$ is equipped with continuous exterior differential
$\{V,\eta_i,d\eta_i\}_{i=1}^k$  at $p_0$. We say that $\Delta$ is non-integrable at $p_0$ if
$$
(\eta_1 \wedge \eta_2 \dots \wedge \eta_k \wedge d\eta_{\ell})_{p_0} \neq 0,
$$
for some $\ell \in \{1, \dots ,k\}$.
\end{defn}

Note that if the bundle $\Delta$ was $C^1$ and corank $1$, then this condition would imply that $\Delta_{p_0} + [\Delta_{p_0},\Delta_{p_0}] =T_{p_0}M$, and therefore would be a step 2, completely non-integrable subbundle at $p_0$.
Therefore the corank $1$, non-integrable, continuous bundles defined above can be thought as of continuous analogues of these step $2$, completely non-integrable subbundles. First of our theorems is the analogue of the Chow-Rashevskii Theorem for continuous bundles whose annihilators are equipped with continuous exterior differentials:

\begin{thm}\label{thm-chow} Let $\Delta$ be a corank 1, continuous tangent subbundle. Let $p_0 \in M$ and
assume  $\mathcal{A}^1_0(\Delta)(M)$ is equipped with a continuous exterior differential $\{V,\eta,d\eta\}$ at
$p_0$. If $\Delta$ is non-integrable at $p_0$ then it is accessible in some neighborhood of $p_0$.
\end{thm}

A direct corollary is

\begin{cor}  Let $\Delta$ be a corank 1, continuous tangent subbundle defined on a connected manifold $M$.
If $\mathcal{A}^1_0(\Delta)(M)$ is equipped with a continuous exterior differential and non-integrable at every $p
\in M$, then $\Delta$ is accessible on $M$. \end{cor}

Our next theorem will be about metric properties of such a bundle, namely we will give an analogue of
the Ball-Box Theorem (specialized to case of differentiable step 2, completely non-integrable tangent subbundles). Theorem \ref{thm-chow} will then be a consequence of this theorem. We say that a bundle
$\Delta$ has modulus of continuity $\omega: s \rightarrow \omega_s$ if in every coordinate neighborhood there exists a constant $C>0$ such that, it has a basis of
sections $\{Z_i\}_{i=1}^n$ whose elements have modulus of continuity ${C}\omega_s$ with respect to the Euclidean norm of the coordinates. More explicitly, these basis of sections satisfy, in coordinates,
$$
|Z_i(p)-Z_i(q)| \leq C\omega_{|p-q|},
$$
for all $i=1,\dots,n$. We assume that moduli of continuity are
increasing and therefore $\omega_t = \sup_{s\leq t}\omega_s$.

Now we remind the notion of adapted coordinates that was introduced informally in the beginning of this section:
\begin{defn}\label{def-adaptedcoordinates}Given a bundle $\Delta$, a coordinate system $(x^1,\dots ,x^n,y^1,\dots,$ $y^m)$ with a domain $U$ around $p\in M$ is said to be adapted to $\Delta$ if $p=0$, $\Delta_0= \text{span}<\frac{\partial}{\partial
x^i}|_0>_{i=1}^n$ and $\Delta_q$ is transverse to  $\text{span}<\frac{\partial}{\partial
y^i}|_q>_{i=1}^m$ for all $q \in U$.
\end{defn}

Then, given some adapted coordinates $(x^1,\dots,x^n,y)$ and the adapted basis for the corank $1$ bundle $\Delta$, we define:
$$
D^{\omega}_{2}(0,K_1,\epsilon)=\{z=(x,y) \in U \quad | \quad  |x| + \sqrt{K_1(|x|\tilde{C}\omega_{2|x|} +|y|)} \leq \epsilon,
$$
$$
H^{\omega}_{2}(0,K_2,\epsilon)=\{z=(x,y) \in U \quad | \quad  |x| \leq \epsilon, \quad |y| \leq K_2\epsilon^{2}+
|x|\tilde{C}\omega_{2|x|}\},
$$
$$
\mathcal{B}_{2}(0,K_3,\epsilon)=\{z=(x,y) \in U \quad | \quad |x| \leq \epsilon, \quad |y| \leq K_3\epsilon^{2}\},
$$
where $|\cdot|$ is the Euclidean norm given by the coordinates and $|x| = \sum_{i=1}^n |x_i|$. Here $H$ stands for hourglass and $D$ for diamond. Note that in the case $\omega_t = t^{\theta}$ (i.e. H\"older continuous)
the shape of $H_2^{\omega}$ indeed looks like an hourglass and $D^{\omega}_{2}$ looks like a diamond with sides that are bent inwards and becomes more linear as $x$ increases (since $K_1^{\frac{1}{2}}|x|^{\frac{1+{\theta}}{2}}$ dominates $|x|$ near $0$, see figure \ref{fig-hourglass}).The ball $\mathcal{B}_{2}(0,K_3,\epsilon)$ is an analogue of the usual box in smooth sub-Riemannian geometry with the exception $y$ direction is allowed to have its own scaling factors $K_3$. We belive that with a careful geometric analysis, these constants turn out to have geometric significance and that is why we decided to define a more generalized ball like this.

In an adapted coordinate system with a domain $U$, we can also define a basis of sections $\Delta$ of the form
$$
X_i = \frac{\partial}{\partial x^i} + a_i(x,y)\frac{\partial}{\partial y},
$$
where $a_i(x,y)$ are continuous functions. If $\Delta$ has modulus of continuity $\omega$, then it is possible to show that, the functions $a_i$ also have modulus of continuity $\tilde{C}\omega$ on $U$ with respect to $|\cdot|$, with some multiplicative constant $\tilde{C}>1$ possibly depending on $U$ and on the chosen coordinates. The assumption of non-integrability at $p_0$ would then mean that
there exists $i,j \in \{1,\dots,n\}$, $i \neq j$ and a domain $U$ such that $|d\eta(X_i,X_j)|_{\inf}>0$. Therefore we also have
$
m(d\eta|_{\Delta})_{\inf}>0.
$
We will also need to define a constant $c$. This constant is later explained in Lemma \ref{lem-Gromov} (and the remark \ref{rem-Gromov} that follows), which is an independent Lemma from \cite{Gro83} (sublemma 3.4B, see also Corollary $2.3$ in \cite{Sim10}).  Finally for a fixed adapted coordinate system with its Euclidean norm, we let $d_g\geq 1$ denote a constant such that, for all $v \in \Delta_p$ and for all $p \in U$:
$$
\frac{1}{d_g}\sqrt{g(v,v)} \leq |v| \leq d_g \sqrt {g(v,v)}.
$$
Then, we can state the next main theorem:

\begin{thm}\label{thm-ballbox} Let $\Delta$ be a corank 1, continuous bundle with modulus of continuity $\omega$. Let $p_0 \in M$ and assume
$\mathcal{A}^1_0(\Delta)(M)$ is equipped with a continuous exterior differential $\{V,\eta,d\eta\}$ at $p_0$ and
that $\Delta$ is non-integrable at $p_0$. Then, for any adapted coordinate system, there exists a domain $U$ and constants $\epsilon_0,K_1,K_2>0$ such that, for all $\epsilon < \frac{\epsilon_0}{2nd_g}$:
\begin{equation}\label{eq-inclusions1}
D^{\omega}_2(0, \frac{1}{K_1},\frac{1}{4d_g}\epsilon) \subset B_{\Delta}(0,\epsilon) \subset H^{\omega}_2(0,K_2, 2nd_g\epsilon),
\end{equation}
where $K_1,K_2>0$ are constants given by
$$
\frac{1}{K_1} = 42\frac{|\eta(\partial_y)|_{\infty}}{m(d\eta|_{\Delta})_{\inf}},
$$
$$
K_2  = 42(1+2n)^2c \frac{ |d\eta|_{\Delta}|_{\infty}}{|\eta(\partial_y)|_{\inf}}.
$$
Moreover for each such smooth adapted coordinate system, there exists a $C^1$ adapted coordinate system such that,
$$
\mathcal{B}_2(0,K_1, \frac{\epsilon}{4d_g}) \subset B_{\Delta}(0,\epsilon) \subset \mathcal{B}_2(0,K_2,2nd_g\epsilon).
$$

\end{thm}

\begin{rem}  This theorem is a generalization of the smooth Ball-Box Theorem on codimension 1, step 2, completely non-integrable bundle case. Indeed if $\Delta$ is smooth then
$\mathcal{A}^1_0(\Delta)(M)$ is equipped with continuous exterior differential and the non-integrability definition given in \ref{defn-integrable} coincides with the step 2, completely non-integrable case.
Also since $\omega_{2|x|} = 2|x|$, one can check the following:
$$
\mathcal{B}_2(0,K_1,\epsilon) \subset D_2^{\omega}(0, \frac{1}{K_1}, (1+\sqrt{\frac{2\tilde{C}}{K_1}+1}) \epsilon),
$$
$$
H_2(0,K_2,\epsilon) \subset \mathcal{B}_2(0,K_2+2\tilde{C},\epsilon).
$$
So one gets in the case of $C^1$ bundles:
$$
\mathcal{B}_2(0,\frac{1}{K_1}, (1+\sqrt{\frac{2\tilde{C}}{K_1}+1})^{-1}\frac{\epsilon}{4d_g}) \subset B_{\Delta}(0,\epsilon) \subset  \mathcal{B}_2(0,K_2+2\tilde{C},2nd_g\epsilon),
$$
which is the usual Ball-Box relations in the smooth sub-Riemannian geometry (apart from the fact that we use a generalized version of the usual boxes which contain and are contained in usual boxes with different constants). Note also that if $\Delta$ is a Lipschitz continuous bundle then again we have that $\omega_t =t$. Therefore this theorem also says that if a Lipschitz continuous bundle has an annihilator equipped with a continuous exterior differential, then the usual Ball-Box relations also hold true for this bundle.

We also would like to note that the statement about the existence of $C^1$ adapted coordinates has a much more geometric interpretation. However we can only explain it after certain objects are constructed. This is explored in subsubsection \ref{ssection-spread} (see figure \ref{fig-ballbox}).
\end{rem}

\begin{rem}The smooth sub-Riemannian geometry is usually shy of explicit constants and the explicit description of the neighbourhood $U$ for the Ball-Box Theorem.  This makes the results particularly hard to apply on a sequence of $C^1$ bundles,
which might be used to approximate a continuous bundle. Therefore we believe that this theorem can also be seen as a version of the Ball-Box Theorem with explicit constants. The explicit constants by themselves are not enough however, but it is also essential to understand how the size of $U$ depends on regularity properties of $\Delta$. The explicit relations are listed in subsection \ref{sssection-fixU}. As far as we are aware this is one of the few proofs that pays particular attention to these details.

\end{rem}

\begin{figure}
  \centering
  \includegraphics[width=300px]{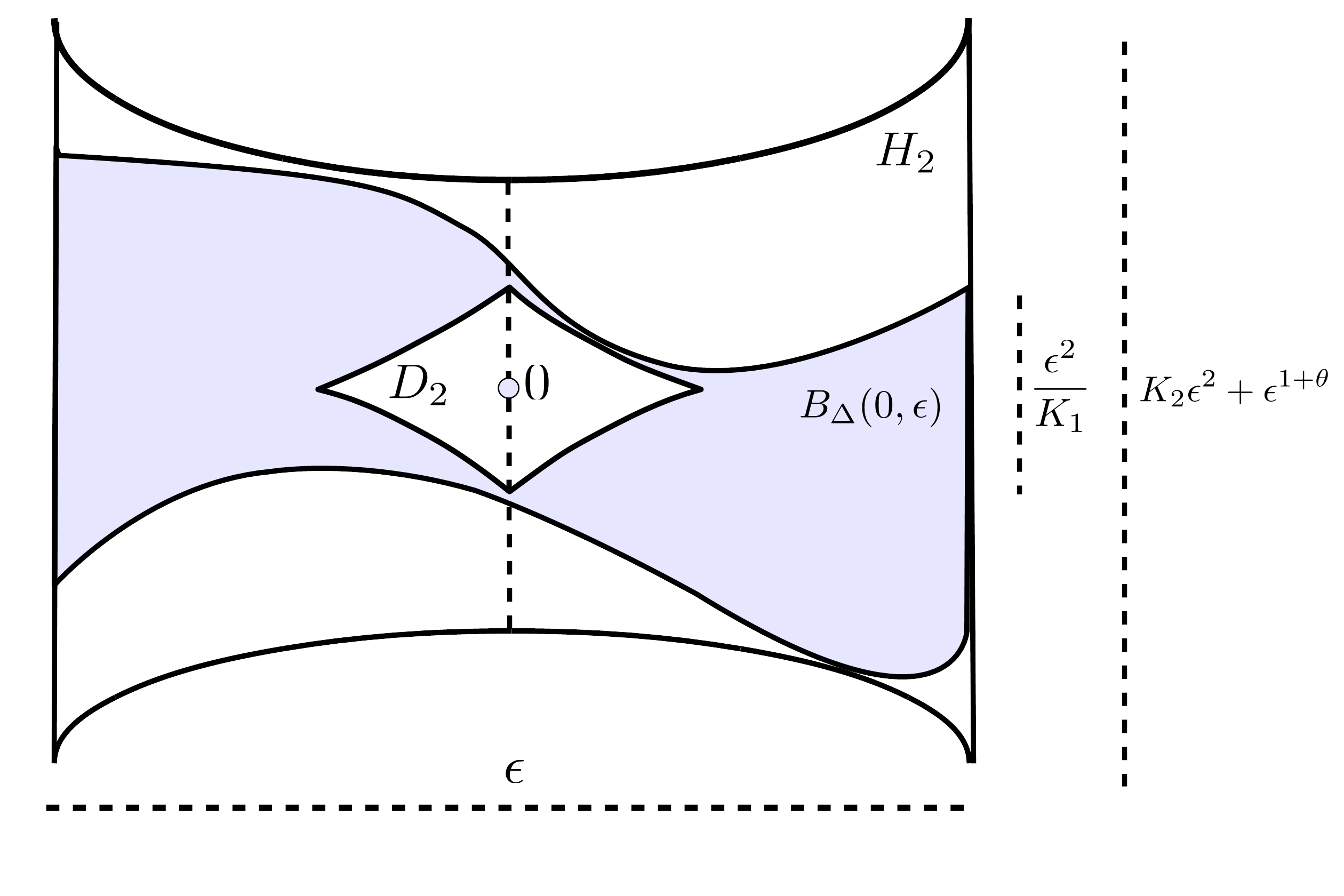}
  \caption{Pictorial representation of Theorem \ref{thm-ballbox} with
  $\omega(t)=t^{\theta}$.}\label{fig-hourglass}
\end{figure}

\subsection{Organization of the Paper}
In this subsection we describe the layout of the paper and the main ideas of the proofs.

First note that Theorem \ref{thm-ballbox} implies  Theorem \ref{thm-chow} therefore it is sufficient to prove the former. Section \ref{sec-proof} contains the proof of Theorem \ref{thm-ballbox}.
The  proof of this theorem has two main ingredients. These are the fundamental tools that we use repeatedly and are therefore are proven in a seperate subsections  \ref{ssec-uniqueness} and \ref{ssec-geometric}.
First ingredient is Proposition \ref{prop-uniquebasis}, where we prove that the adapted basis is uniquely integrable. This is
only due to existence of a continuous exterior differential. The second ingredient is Proposition \ref{prop-stokessubriemann} which quantifies the amount admissible curves travel in the $\partial_y$ direction by a certain surface integral of $d\eta$. This again only assumes existence of the continuous exterior differential and Proposition \ref{prop-uniquebasis}. This proposition can be seen as a generalization of certain beautiful ideas from \cite{Arn89} (see Section 36 of Chapter 7 and Appendix 4).

Then, in subsection \ref{subsec-ballboxproof} we prove Theorem \ref{thm-ballbox} using Propositions \ref{prop-uniquebasis} and \ref{prop-stokessubriemann}. The main idea is to first construct certain, accessible $n$ dimensional manifolds $\mathcal{W}_{\epsilon}$ (see Lemma \ref{lem-propertiesofW}) and study how the sub-Riemannian balls are spread around these manifolds (see Lemma \ref{lem-distancefromW}).

To summarize the main ideas of the proof of Theorem \ref{thm-ballbox}, we can say

\begin{itemize}
  \item The existence of $d\eta$ and the condition that $\eta \wedge d\eta \neq 0$ gives volume to the
      sub-Riemannian ball.
  \item The loss of regularity in the bundle may cause the sub-Riemannian ball to bend which results
  in the outer sub-Riemannian ball being distorted and the inner one getting smaller.
\end{itemize}

The two main tools that we repeatedly use (Proposition  \ref{ssec-uniqueness} and \ref{ssec-geometric}) are obtained via application of Stokes property,
thus establishing it as the germ of many geometric and analytic properties of vector fields and bundles.

In section \ref{sec-continuousexteriordifferential}, we give some examples of bundles whose annihilators
are equipped with continuous exterior differentials and which are non-integrable on neighbourhoods where they are non-differentiable. We also study some of the
properties of such bundles to emphasize that having a continuous exterior differential is a
geometrically very relevant property, yet not as strong as being $C^1$ in terms of regularity. We compare the results of this paper to the several other results we have explained in the introduction.

Finally in section \ref{sec-generalizations} we sketch some thoughts on some possible generalizations that relax the
conditions required for the theorems and some comments on how to possibly proceed with the proof in special cases of higher corank bundles.

\section{The Proof} \label{sec-proof}

In the next two subsections we prove the two technical propositions: Proposition  \ref{prop-uniquebasis} and Proposition \ref{prop-stokessubriemann}.

\subsection{ Proposition \ref{prop-uniquebasis}}\label{ssec-uniqueness}

Lets remind the definition of $\{X_i\}_{i=1}^n$. Given $p_0 \in M$, assume we are given any adapted coordinates $\{x^1,\dots,x^n,y\}$ with some domain $U$ such that $\frac{\partial}{\partial y}$ is
everywhere transverse
to $\Delta$ on $U$.  Occasionally we will denote $\partial_i = \frac{\partial}{\partial x_i}$, $\partial_y = \frac{\partial}{\partial y}$, $\mathcal{X}^k_p= \text{span}\frac{\partial}{\partial x^k}|_p$ and $\mathcal{Y}_p= \text{span}\frac{\partial}{\partial y}|_p$ . Then, it is easy to show that in this
domain, sections of $\Delta$ admit a basis of the form
\begin{equation}\label{eq-adapted}
X_i = \frac{\partial}{\partial x^i}+a_i(x,y)\frac{\partial}{\partial y},
\end{equation}
where $a_i$ have the same modulus of continuity as $\Delta$ up to multiplication by some constant $\tilde{C}>0$.
Note that the adapted coordinate assumption also means $X_i(0)=\partial_i$. We call such a basis an adapted basis. 

\begin{rem}It is also easy to show that such an adapted basis also exists in
higher coranks but of the form
$$
X_i = \frac{\partial}{\partial x^i}+\sum_{j=1}^m a_{ij}(x,y)\frac{\partial}{\partial y^j}.
$$
\end{rem}

If $X$ is any vector-field defined on some neighborhood $U_1$, we call it integrable if for all $p \in U_1$ there exists
$\epsilon_p$ and a $C^1$ curve $\gamma: [-\epsilon_p, \epsilon_p] \rightarrow U_1$ such that $\gamma(0)=p$
and $\dot{\gamma}(t) = X(\gamma(t))$ for all $t \in  [-\epsilon_p, \epsilon_p]$ (these curves are called
integral curves passing through $p$). By Peano's Theorem, continuous vector-fields are always integrable.

We call it uniquely integrable if it is integrable and if $\gamma_1$ and $\gamma_2$ are two integral curves which intersect, then each intersection
point is contained in a relatively open (in both integral curves) set.
In this case, there exists a unique maximal integral curve of $X$ (not to be confused with maximal and minimal solutions of non-uniquely integrable vector-fields), starting at $q$ and defined on the interval $[-\epsilon_q,\epsilon_q]$. We denote this integral curve by $t \rightarrow e^{tX}(q)$.
Unique integrability is more stringent
and commonly known sufficient conditions are $X$ being Lipschitz or $X$ satisfying the Osgood criterion.

We say that an integrable vector-field $X$ defined on $U_1$ has $C^r$ family of solutions if there exists some $U_2 \subset U_1$, an $\epsilon_0$ such that, for all $p \in U_2$,
there exists an integral curve passing through $p$ with $\epsilon_p \geq \epsilon_0$ and such that this choice of solutions seen as maps from $[-\epsilon_0,\epsilon_0] \times U_2 \rightarrow U_1$ are $C^r$.

\begin{prop}\label{prop-uniquebasis} Assume $\Delta$ is a co-rank $1$, continuous tangent subbundle such that,
at $p_0$, $\mathcal{A}^1_0(\Delta)(M)$ is equipped with a continuous exterior differential $\{V,\eta, d\eta\}$. Then, for any adapted coordinate
$\{x^1,\dots$ $,x^n,y\}$ around $p_0$ with an adapted basis $\{X_i\}_{i=1}^{n}$, there exists some domain $U$ on which $X_i$ are uniquely integrable
and have $C^1$ family of solutions.
\end{prop}

\begin{proof}  Pick any adapted coordinate system and adapted basis with some domain $U \subset V$.
Assume by contradiction that there exists an $X_k$ which is not uniquely integrable. The there exists two integral curves $\gamma_i:(0,\epsilon_i) \rightarrow U$ which intersect
at some point $p=\gamma_i(\tau_i)$ which is not contained in a relatively open set in one of the curves. This means there exists an interval
$[\tau_i,\kappa_i]$ (WLOG assume $\kappa_i > \tau_i$) on which $\gamma_i$ do not coincide but are defined and such that $\gamma_1(\tau_1)=\gamma_2(\tau_2)$. By shifting
and restricting to a smaller interval we can then assume we have $\gamma_i: [0, \epsilon_0] \rightarrow U$ such that $\gamma_1(0) = \gamma_2(0)$
but that they are not everywhere equal. Now take any $q = \gamma_1(\epsilon_1)$ that is not in $\gamma_2$. Therefore they also do not coincide on some interval
around $\epsilon_1$. Denote $\epsilon_2= \sup_{0 \leq t \leq \epsilon_1}\{t \quad \text{such that} \quad \gamma_1(t)=\gamma_2(t)\}$. $\epsilon_2$ exists since we know at least that $\gamma_1(0)=\gamma_2(0)$.
We have that clearly $\epsilon_2 < \epsilon_1$ and between $\epsilon_2$ and $\epsilon_1$ $\gamma_1,\gamma_2$ can not intersect. So by restricting everything to $[\epsilon_2,\epsilon_1]$ and reparametrizing
we obtain two integral curves of $X_k$ defined on some $[0,\epsilon]$ such that $\gamma_1(0)=\gamma_2(0)$
and that $\gamma_1(t) \neq \gamma_2(t)$ for all $0<t\leq \epsilon$ for some $\epsilon$. Due to the specific
form of $X_k$ both curves lie in the $\mathcal{X}^k_p-\mathcal{Y}_p$ plane (whose coordinate we will denote
as $(x,y)$) and have the form:
$$
\gamma_j(t) = (t, d_j(t)).
$$
Without loss of generality we can assume $d_1(t)> d_2(t)$ for all $0<t\leq \epsilon$. We are going to show
that existence of exterior differential forces $d_1(t)=d_2(t)$ for all $t \leq \epsilon$ leading to a
contradiction. To this end let $h(t) = d_1(t)-d_2(t)$.

Before continuing with the proof we make one elementary remark. By our choice of coordinates $\eta$ will have the form
$$
\eta = a_0(x,y)dy + \sum_{i=1}^n a_i(x,y)dx^i,
$$
with $\inf_{q \in U}|a_0(q)|= |\eta(\partial_y)|_{\inf}>0$ (since $\Delta$ is always transverse $y-$direction, it can not contain $\partial_y$ and therefore $b$ can not be $0$) and so in particular $a_0(q)$ has
constant sign. So if $\alpha(t)$ is any (non-singularly parametrized) curve whose tangent vectors lie in $\mathcal{Y}_{\gamma(t)}$ axis, one has that
$\eta(\dot{\alpha})$ always has constant sign and therefore
$$
\bigg|\int_{\gamma}\alpha \bigg| = \int_{\gamma}\bigg|\alpha\bigg| \geq |\eta(\partial_y)|_{\inf}| \ |\alpha|,
$$
where $|\alpha|$ is the Euclidean length of the curve $\alpha$.

Let $v_t$ be the straight line segment that lies in the $\mathcal{Y}_{d_2(t)}$ axis and which starts at
$d_2(t)$ and ends at $d_1(t)$. We let $\gamma_t$ be the loop that is formed by composing $\gamma_2, v_t$
and then $\gamma_1$ backwards. We also let $\Gamma_t$ be the surface in
$\mathcal{X}^k_p-\mathcal{Y}_p$ plane that is bounded by this curve. Note that since  $\eta(\dot{\gamma}_i)=0$,
$\int_{\gamma_t}\eta = \int_{v_t}\eta$. But $\dot{v}_t=\partial_y$ and $\partial_y$ is always transverse to $\Delta$ on $U$.
Therefore $\eta(\dot{v}_t)$ is never $0$ and so it never changes sign. So we have that
\begin{equation}\label{eq-segmentlength}
\bigg|\int_{\gamma_t}\eta\bigg| = \bigg|\int_{v_t}\eta\bigg| > |\eta(\partial_y)|_{\inf} h(t).
\end{equation}
Since $\eta$ has the continuous exterior differential $d\eta$, we have using Stokes property and equation \eqref{eq-segmentlength}
$$
 \bigg|\int_{\Gamma_t}d\eta\bigg|=\bigg|\int_{\gamma_t}\eta\bigg|  \geq |\eta(\partial_y)|_{\inf} h(t).
$$
But $|\int_{\Gamma_t}d\eta| \leq |\Gamma_t||d\eta|_{\infty}$ (where $|\Gamma_t|$ denote the Euclidean area). Therefore we have
\begin{equation}\label{eq-inequni}
h(t) \leq \frac{|d\eta|_{\infty}}{|\eta(\partial_y)|}|\Gamma_t|.
\end{equation}
We are going to show that this leads to a contradiction for $t$ small
enough. Let $t_n$ be a sequence of times such that $t_n \leq \epsilon$, $t_n \rightarrow 0$ as
$n\rightarrow \infty$ and
\begin{equation}\label{eq-assumption}
h(t_n) \geq \sup_{t<t_n}h(t),
\end{equation}
which is possible since $h(t)$ is continuous and $0$ at $0$. Let $S_n$ be the strip obtained by parallel sliding the
segment $v_{t_n}$ along the curve $\gamma_1$. By our assumption in equation \eqref{eq-assumption}, $S_n$
contains the surface $\Gamma_{t_n}$ (see figure \ref{fig-strip}).

\begin{figure}
  \centering
  \includegraphics[width=150px]{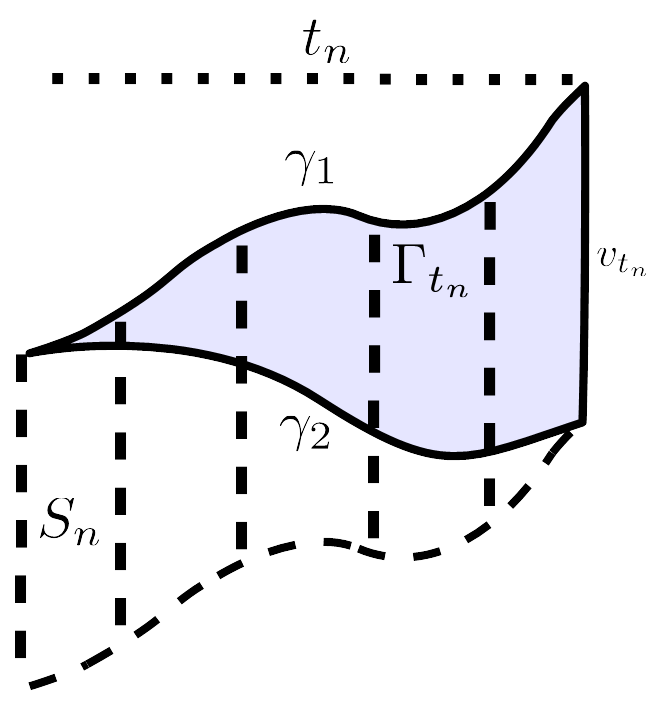}
  \caption{Strip $S_n$.}\label{fig-strip}
\end{figure}

\begin{lem}\label{lem-striparea}$|S_n|= t_n h(t_n)$.\end{lem}
\begin{proof} Consider the transformation $(x,y)\rightarrow (x, y-\gamma_1(x))$ on its maximally defined
domain (which includes $S_n$ and which is differentiable since $\gamma_1(t)$ is differentiable in the $t$ variable). It takes the strip $S_n$ to a rectangle with two sides of length $t_n$ and
$h(t_n)$ so in particular it has area $t_n h(t_n)$. The Jacobian of this transformation is $1$ and
therefore it preserves area so the strip itself has area $t_n h(t_n)$.
\end{proof}
Since this strip contains $\Gamma_n$ we see that $|\Gamma_n| \leq t_n h(t_n)$. Then, using equation
\eqref{eq-inequni} we obtain for all $t_n$
$$
h(t_n) \leq \frac{|d\eta|_{\infty}}{|\eta(\partial_y)|_{\inf}} t_n h(t_n) \Rightarrow
\frac{|\eta(\partial_y)|_{\inf}}{|d\eta|_{\infty}} \leq t_n,
$$
which leads to a contradiction since $t_n$ tends to $0$. This concludes the proof of uniqueness.

For being $C^1$, note that we can always find a neighborhood $\tilde{U} \varsubsetneq U$ and some $\epsilon_0$ such that
for all $q \in \tilde{U}$ and for all $|t| \leq \epsilon_0$, $e^{tX_k}(q) \in U$ and is well defined (the size of $\tilde{U}$ and $\epsilon_0$ depend on each other and on $|X_k|_{\infty}$). Therefore we obtain the map
$[-\epsilon_0, \epsilon_0] \times \tilde{U} \rightarrow U$.
Then, for being $C^1$ we use a result from \cite{Har64} (the theorem stated there is more general so we state the specialized version)\footnote{One can prove that the solutions are $C^1$ using Stokes Theorem
on a sequence of approximations $\eta^k$ built in a certain way but it gets very lengthy and technical.}:
\begin{thm} Let $f(t,y): U\subset \mathbb{R}^{n+1} \rightarrow \mathbb{R}^n$ (with $y \in \mathbb{R}^n$) be continuous. Then, the ODE $\dot{y}=f(t,y)$ has unique and $C^1$ solutions
$y= \eta(t,t_0,y_0)$ for all $(t_0,y_0) \in U$ if for any $p \in U$, there exists a neighborhood $U_p$ and a non-singular matrix $A(t,y)$ such that the $1-$forms
$\eta^i = \sum_{i=1}^n A^{ij} (dy^i - f^idt)$ have continuous exterior differentials.
\end{thm}

Note first that the unique integrability of the non-autonomous ODE with $C^1$ solutions above is equivalent to unique integrability of the vector-field $X=\frac{\partial}{\partial t} + \sum_{i=1}^n f^i(t,y,z)\frac{\partial}{\partial y^i}$ with $C^1$ family of integral curves which would be given by $t \rightarrow (t, \eta(t,t_0,y_0)) \subset U$ for $t$ small enough. Second let $\mathbb{X}$ be the bundle spanned by this vector-field in the $(t,y)$ space. Then, this bundle is the intersection of the kernel of the 1-forms $\eta^i=dy^i - f^i dt$. Therefore
$\eta^i$ is a basis of sections for $\mathcal{A}_0^1(\mathbb{X})$. In particular then the condition of this theorem about the existence of such a non-singular $A(t,y)$ simply means that there must exist a basis  of sections for $\mathcal{A}_0^1(\mathbb{X})$ which has continuous exterior differentials. In our case for each $X_k$ we have explicitly built that basis of sections which is given by  $\eta, dx^1,\dots,dx^{k-1},dx^{k+1},\dots,dx^n$.

\end{proof}
\begin{rem}It is possible to prove stronger versions of this theorem but they use approximations to
$\Delta$ rather than $\Delta$ itself and $C^1$ regularity is interchanged with Lipschitzness. We refer the
interested readers to \cite{Har64, LuzTurWar16} for the generalizations.
\end{rem}

\subsection{Proposition \ref{prop-stokessubriemann}}\label{ssec-geometric}

For the following, given some $p \in U$, let $\bar{X}_p$  be the space spanned by $X_i(p)$ at the point $p$ , $|X|_{\infty}=\max_{i=1,\dots, n}|X_i|_{\infty}$ , $|\wedge X|_{\inf} =   |X_1 \wedge \dots \wedge X_n \wedge \partial_y|_{\inf}$ and $\Pi_x$ be the projection to the $x$ coordinates.

\begin{prop}\label{prop-stokessubriemann} Let $\gamma_i: [0,\varepsilon_i] \rightarrow U$ for $i=1,2$ be two $\Delta-$admissible curves with lengths $\ell_{i}$ such that
$\gamma_1(0)=\gamma_2(0)=q$, $\Pi_x\gamma_1(\varepsilon_1) = \Pi_x\gamma_2(\varepsilon_2)$ (that is they start on the same point and end at the same $\partial_y$ axis). Let $\ell = \max\{\ell_1,\ell_2\}$, $\varepsilon=\max\{\varepsilon_1,\varepsilon_2\}$, $\xi = \max_{k=1,2} n \frac{(n|X|_{\infty})^{n}}{|\wedge X|_{\inf}}\sup_{t \leq \varepsilon_k} |\dot{\gamma}_k(t)|\tilde{C}\omega_{\ell} $ and $\beta$ be the segment in the $\partial_y$ direction that connects $\gamma_1(\varepsilon_1)$ to $\gamma_2(\varepsilon_2)$. Assume moreover that $B(q, 2\ell) \subset U$.
Then, for any 2-chain $P \subset \bar{X}_{q} \cap U$ whose boundary is the projection of $\gamma^{-1}_1 \circ \gamma_2$ along $\partial_y$ to $\bar{X}_{q}$ we have that
\begin{equation}\label{eq-mainprop}
\frac{1}{|\eta(\partial_y)|_{\infty}}\bigg(\bigg|\int_P d\eta \bigg| - |c|\bigg) \leq |\gamma_1(\varepsilon_1) - \gamma_2(\varepsilon_2)| \leq \frac{1}{|\eta(\partial_y)|_{\inf}}\bigg(\bigg|\int_P d\eta\bigg| + |c|\bigg),
\end{equation}
and
\begin{equation}\label{eq-mainprop2}
sign\bigg(\int_{\beta}\eta\bigg) = sign\bigg(\int_P d\eta + c\bigg),
\end{equation}
where
\begin{equation}
|c| \leq 4 \ell \varepsilon \xi |d\eta|_{\infty}.
\end{equation}
\end{prop}

\begin{proof} Denote $\gamma_1(\varepsilon_1)=q_1$, $\gamma_2(\varepsilon_2)=q_2$. Assume wlog that $q_1 \geq q_2$ with respect to the order given by the positive orientation of $\partial_y$ direction.
We first define the projections if $\gamma_i$ to $\bar{X}_q$. Since $\gamma_i$ are admissible curves we have that
$$
\dot{\gamma}_i(t)= \sum_{k=1}^n u_i^k(t)X_k(\gamma_i(t)) = \sum_{k=1}^n u_i^k(t)(\partial_k + a_k(\gamma_i(t))\partial_y) ,
$$
$t$ a.e for some piecewise $C^1$ functions $u_i^k(t)$. Define the following non-autonomous vector-fields
$$
Z_i(t,p) = \sum_{k=1}^n u_i^k(t)X_k(q)= \sum_{k=1}^n u_i^k(t)(\partial_k + a_k(q)\partial_y),
$$
which are constant in the $p$ variable. Therefore they admit unique solutions starting at $t=0$ and $q$ which we denote as $\alpha_i(t,q,0) = e^{tZ_i}(q)$ that are inside $\bar{X}(q)$. We will denote the images of these curves as $\alpha_{i}$. Its clear that $\Pi_x(\gamma_i(t))=\Pi_x(\alpha_i(t))$. Since $\bar{X}_q$ is transversal to $\partial_y$ direction, these are the unique projections of $\gamma_i$ to $\bar{X}_q$ alogn $\partial_y$ direction. The condition that $B(q,2\ell) \subset U$ also implies that they are inside $U$. We can build some 2 chains inside $U$ bounded by these 1 chains as follows \footnote{We remind that a $n$ cell in $U$ is a differentiable mapping from a convex n-polyhedron in $\mathbb{R}^n$ (with an orientation) to $U$ and a $n$ chain is a formal sum of $n$ cells over integers.}
$$
v_{\ell}(t,s) = \alpha_{\ell}(s) + t(\gamma_{\ell}(s)-\alpha_{\ell}(s)),
$$
from $[0,1]\times [0,\epsilon]$ to $U$. Since $\alpha_{\ell}(s)$ and $\gamma_{\ell}(s)$ are piecewise $C^1$ in the $s$ variable, the domain of this map can be partitioned into smaller rectangles on which $v_{\ell}(t,s)$ are differentiable and therefore whose images are $2$ cells. Then, the images of $v_{\ell}(t,s)$ become 2 chains which we denote as $C_{\ell}$.  Note that $v_{1}(t,0) = v_{2}(t,0)=q$ for all $t$. And also let the image of $v_{2}(t,\varepsilon_2)$ be a curve $\tau$. Since $q_1 \geq q_2$ then image of  $v_{1}(t,\varepsilon_2)$ is $\beta \circ \tau$.
It is also clear that $v_{\ell}(0, s) = \alpha_{\ell}(s)$ and $v_{\ell}(1,s)=\gamma_{\ell}(s)$.
Then, we orient these curves and $C_i$ such that:
$$
\partial C_1 =   \gamma_1 -\beta  - \tau   -\alpha_1,
$$
$$
\partial C_2 = \tau -\gamma_2 + \alpha_2.
$$

Let $\Gamma$ be any 2 chain in $U$ bounded by concatenating $\gamma_1, \gamma_2$ and $\beta$ (whose composition is a 1 cycle in a contractible space so it always bounds a chain) in the right orientation so that
$$
\partial \Gamma =  \beta  -\gamma_1 + \gamma_2.
$$
Finally also orient $P$ so that
$$
\partial P = \alpha_1  - \alpha_2.
$$
Then, $\Gamma, C_1,C_2$ and $P$ form a closed 2 chain $CC$ (see figure \ref{fig-cycle}). Using Stokes property and the fact that $\partial CC = \emptyset$ we get

\begin{figure}
  \centering
  \includegraphics[width=150px]{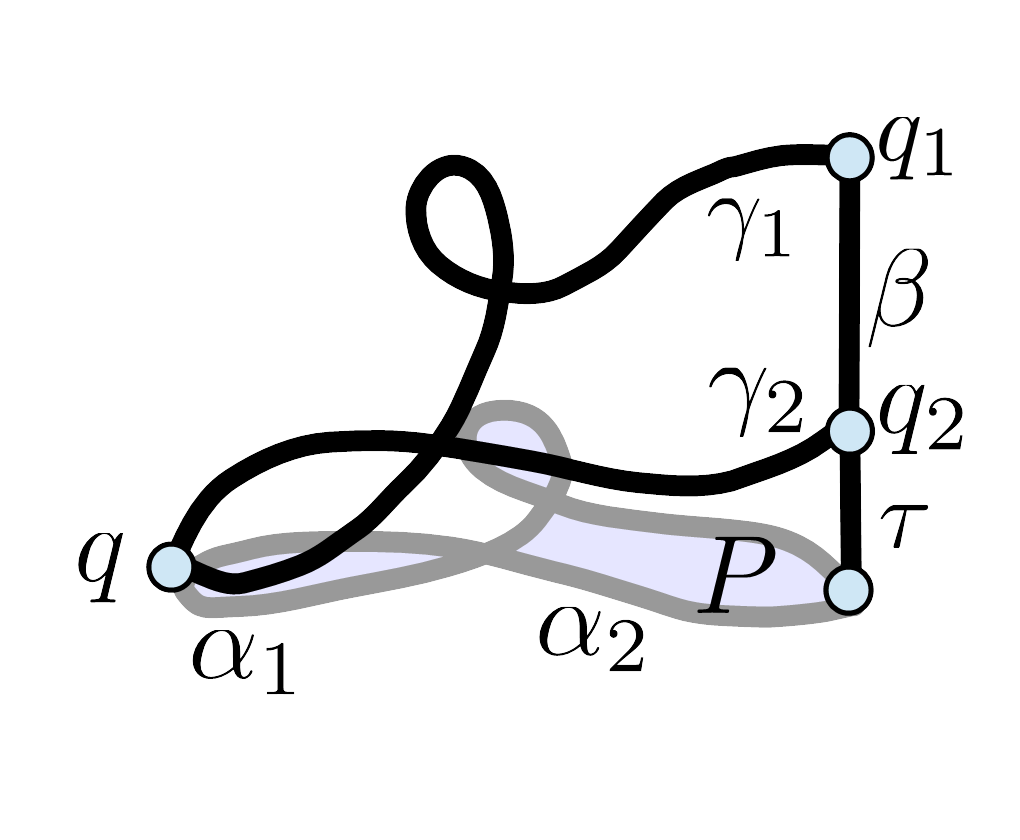}
  \caption{The Closed 2-Chain C.}\label{fig-cycle}
\end{figure}

$$
\int_{\Gamma}d\eta = -\bigg(\int_{P}d\eta + \int_{C_1 \cup C_2}d\eta \bigg).
$$
Moreover since $\Gamma$ is bounded by $\gamma_1,\gamma_2$ and $\beta$ and $\eta_{\gamma_i(t)}(\dot{\gamma}_i(t))=0$ we get again by Stokes property
\begin{equation}\label{eq-finaleq}
 \int_{\beta}\eta=\int_{\Gamma}d\eta  = -\bigg(\int_{P}d\eta + \int_{C_1 \cup C_2}d\eta \bigg).
\end{equation}
Defining $c= \int_{C_1 \cup C_2}d\eta$,  we require one more final lemma to finish the proof,

\begin{lem} For $\beta,c$ as defined above
\begin{equation}\label{eq-Cintegral}
|c|  \leq   4\xi\ell \varepsilon|d\eta|_{\infty},
\end{equation}

\begin{equation}\label{eq-Bintegral}
|q_2-q_1||\eta(\partial_y)|_{\inf}  \leq \bigg|\int_{\beta} \eta \bigg| \leq |\eta(\partial_y)|_{\infty}|q_2-q_1|.
\end{equation}

\end{lem}

\begin{proof}
For the first inequality with $\dot{\alpha}(s)$ and $\dot{\gamma}(s)$ a.e defined we can write,
\[
\begin{aligned}
\bigg|\int_{C_i}d\eta \bigg| &= \bigg|\int_0^1 dt \int_0^{\epsilon}ds \ d\eta_{v_i(t,s)} (\gamma_i(s)-\alpha_i(s), (1-t)\dot{\alpha}_i(s) + t\dot{\gamma}_{i}(s)) \bigg|, \\
& \leq  \int_0^1 dt \int_0^{\epsilon}ds \ |d\eta|_{\infty} \sup_{s \leq \epsilon}|\alpha_i(s)-\gamma_i(s)|((1-t)|\dot{\alpha}_i(s) | + t|\dot{\gamma}_i(s)|), \\
& \leq  2\ell|d\eta|_{\infty} \sup_{s \leq \epsilon}|\alpha_i(s)-\gamma_i(s)|.\\
\end{aligned}
\]
Therefore we need to obtain an estimate on  the maximum distance between $\gamma_i$ and $\alpha_i$. For this, we have that
$$
|\alpha_i(t) - \gamma_i(t)| \leq n \varepsilon_i \sup_{t \leq \varepsilon_i, \ell=1,\dots,n}|u_i^{\ell}(t)||X_{\ell}(\gamma_1(t))-X_{\ell}(q)|.
$$
Since $|X_{\ell}(\gamma_i(t))-X_{\ell}(q)| \leq \tilde{C}\omega_{\ell}$ we have that
$$
|\alpha_i(t) - \gamma_i(t)| \leq n \varepsilon_i \sup_{t \leq \varepsilon_i, \ell=1,\dots,n}|u_i^{\ell}(t)|\tilde{C}\omega_{\ell}.
$$
Now we need to estimate $\max_{\ell=1,\dots,n}|u_i^{\ell}(t)|$. Let $L_p$ be the linear transformation that takes $X_i(p)$ to $\partial_i$ and fixes $\partial_y$ (which is a matrix whose columns in the Euclidean basis are $X_i(p)$ and $\partial_y$). Then, $det(L^i_p) \geq  |X_1 \wedge \dots \wedge X_n \wedge \partial_y|_{\inf}$ for all $p\in U$ which is non-zero by the form of $X_i$. Let $s_1(p)=m(L_p),s_2(p),\dots,s_{n+1}(p)=|L_p|$ be the singular values of $L_p$. Then, $det(L_p) = s_1(p)\dots s_n(p) \leq m(L_p)|L_p|^n$. Since $L_p$ is a matrix whose each column has norm less than $|X|_{\infty}>1=|\partial_y|$ one has that $|L_p| \leq n|X|_{\infty}$ for all $p\in U$ so,
$m(L_p) \geq \frac{|\wedge X|_{\inf}}{(n|X|_{\infty})^{n}}$. Then,
$$
|\dot{\gamma}_i(t)| = |\sum_{k=1}^n u^k_i(t)L_{\gamma_k(t)}\partial_k| \geq m(L_{\gamma_i(t)})\max_{k=1,\dots,n}|u^k_i(t)|.
$$
So
$$
\sup_{t < \varepsilon_i, k=1,\dots,n }|u^k_i(t)| \leq\frac{(n|X|_{\infty})^{n}}{|\wedge X|_{\inf}}\sup_{t < \varepsilon_i}|\dot{\gamma}_i(t)|.
$$
This gives
$$
|\alpha_i(t) - \gamma_i(t)| \leq \varepsilon n \frac{(n|X|_{\infty})^{n}}{|\wedge X|_{\inf}}\sup_{t < \varepsilon_i}|\dot{\gamma}_i(t)|\tilde{C}\omega_{\ell} \leq \varepsilon \xi.
$$
With this estimate in hand, we have $|\int_{C_i}d\eta| \leq  2\ell \varepsilon \xi$, and so  $|c| =|\int_{C_1 \cup C_2}d\eta| \leq  4\ell \varepsilon \xi$.

For the second note that $\beta$ is a curve whose tangent vector is always parallel to $\partial_y$. But since $\eta$ annihiliates $\Delta$ which is transverse to $\partial_y$, we have that $\eta_p(\partial_y)$ is never $0$ in $U$ and never changes sign. Then, similarly for any non-singular parametrization $\eta_{\beta(t)}(\dot{\beta}(t))$ also never changes sign. So assuming unit length parametrization
$$
\bigg|\int_{\eta}\beta\bigg| = \int_{0}^{|q_2-q_1|}|\eta_{\beta(t)}(\dot{\beta}(t))|dt,
$$
which gives
$$
 |\eta(\partial_y)|_{\inf}|q_2-q_1| \leq \bigg|\int_{\eta}\beta\bigg| \leq |\eta(\partial_y)|_{\infty}|q_2-q_1|.
$$

\end{proof}

Now using equations \eqref{eq-finaleq}, \eqref{eq-Cintegral} and \eqref{eq-Bintegral} we get that
$$
\frac{1}{|\eta(\partial_y)|_{\infty}}\bigg(\int_P d\eta - c\bigg) \leq |q_2 - q_1| \leq \frac{1}{|\eta(\partial_y)|_{\inf}}\bigg(\int_P d\eta + c\bigg),
$$
where
$$
|c|=\bigg|\int_{C_1 \cup C_2}d\eta\bigg|\leq 4\xi\ell\varepsilon|d\eta|_{\infty}.
$$
The claim about the sign is also an immediate consequence of equations  \eqref{eq-finaleq} and \eqref{eq-Cintegral}.

\end{proof}

\subsection{Fixing $U$ and $\epsilon_0$}\label{sssection-fixU}

Now, before carrying out the rest of the proof of Theorem  \ref{thm-ballbox}, we will fix $\epsilon$ and $U$ once and for all. Assume we are given any adapted coordinate system with $p_0 \rightarrow 0$ and an adapted basis. Choose the domain $U$ containing $0$ so that
\begin{itemize}

\item There exist some $i,j \in \{1,\dots,n\}$ such that for all $q \in U$
\begin{equation}\label{eq-noninvolutive}
d\eta_q(X_i,X_j) \neq 0.
\end{equation}
This is possible due to assumption of non-integrability of at $0$ and of the continuity of $\eta$ and $d\eta$. This also implies $m(d\eta|_{\Delta})_{\inf}>0$.

\item For all $\ell,k=1,\dots,n$,
\begin{equation}\label{eq-normX}
\frac{1}{2} \leq |X_{\ell}|_{\inf} \leq |X_{\ell}|_{\infty} \leq 2,
\end{equation}
\begin{equation}\label{eq-normX2}
\frac{1}{2} \leq |X_{\ell} \wedge X_k|_{\inf} \leq |X_{\ell} \wedge X_k|_{\infty} \leq 2,
\end{equation}
and
\begin{equation}\label{eq-normX3}
\frac{1}{1.75} \leq |X_1\wedge \dots X_n\wedge \partial_y|_{\inf} \leq|X_1\wedge \dots X_n\wedge \partial_y|_{\infty} \leq 2.
\end{equation}
These are possible since at the origin $X_{\ell}(0) = \partial_{\ell}$ and so the norms above are equal to $1$ at the point $0$.

\item
The vector-fields $\{X_{\ell}\}_{\ell=1}^{n}$ are uniquely integrable on $U$ and have $C^1$ family of solutions $e^{tX_{\ell}}(q)$ defined on $[-\epsilon_0, \epsilon_0] \times \tilde{U}$
for some $\epsilon_0$ and $\tilde{U} \varsubsetneq U$ containing $0$.

\end{itemize}

Finally fix $\epsilon_0$ so that,

\begin{itemize}

\item For all $q \in \tilde{U}$
\begin{equation}\label{eq-remaininside1}
B((13 + (c+1)(2+6n)d_g)\epsilon_0,q) \subset U,
\end{equation}
where $B(r,q)$ is defined with respect to the Euclidean norm and $c>0$ is the constant appearing in Lemma \ref{lem-Gromov} and the remark \ref{rem-Gromov} that follows. Although this inequality will be used in various places, one immediate consequence that we state now is that
for all $q \in \tilde{U}$, $|t_{i_k}| \leq \epsilon_0$, where $i_k \in \{1, \dots ,n\}$ and
$k=1,\dots ,n+4$, one has that
\begin{equation}\label{eq-remaininside2}
e^{t_{i_k}X_{i_k}}\circ \dots \circ e^{t_{i_1}X_{i_1}}(q) \in U.
\end{equation}
That is starting at $q$, even we apply the flows consecutively $n+4$ times up to time $\epsilon_0$ we still
stay in the domain $U$ (for this $B(q, 2(n+4))$ is sufficient). This also guarantees us that on $[-\epsilon_0,\epsilon_0] \times \tilde{U}$ this composition of flows is a $C^1$ map
with respect to $q$ and $t_{i_k}$.

\item Also the following are satisfied:

\begin{equation}\label{eq-estimate1}
|d\eta|_{\infty}\tilde{C}\omega_{13(c+1)\epsilon_0}(4 + \tilde{C}\omega_{13(c+1)\epsilon_0})(2n)^{n+7}d^2_g   < \frac{1}{4}|d\eta(X_i,X_j)|_{\inf},
\end{equation}

\begin{equation}\label{eq-estimate3}
3\epsilon_0 \leq \delta.
\end{equation}
where $\delta>0$, $c>0$  are the constants appearing in Lemma \ref{lem-Gromov} (and the remark \ref{rem-Gromov} that follows).
These  are possible since $\omega_0=0$ and $\omega_t$ is continuous, and for all $q\in U$ we have that $d\eta_q(X_i,X_j) > 0$.
\end{itemize}

\begin{rem} \label{rem-start}Some of these conditions were chosen for convenience and if one digs carefully into the proof, it can be seen that the numerical constants appearing in the expressions are not sharp. However sharp constants are not relevant for our purposes so in order not to introduce additional complexity to the exposition, we will not attempt to sharpen them. Also we  note that several similar looking conditions were combined into a single condition in equation \eqref{eq-estimate1} using the ''worst'' one, in order to decrease the number of conditions.
\end{rem}

\subsection{Proof of Theorem \ref{thm-ballbox}}\label{subsec-ballboxproof}

The part of the Theorem  \ref{thm-ballbox} about smooth adapted coordinates is divided into two seperate parts. First given $\epsilon$, we will construct a certain $n-$dimensional manifold $\mathcal{W}_{\epsilon}$ using the adapted basis $X_i$, which is admissible and transverse to $\partial_y$ direction. This is carried out in subsubsection \ref{ssec-partI}.
Then, next we will describe how the sub-Riemannian ball spreads around these manifolds in subsubsection \ref{ssection-spread}. The results we obtain in these sections will then quickly lead us to the proof of the theorem. The construction of $\mathcal{W}_{\epsilon}$ will be based on the Proposition \ref{prop-uniquebasis}
while the description of the spread of the sub-Riemannian ball will use Proposition \ref{prop-stokessubriemann}. The proofs will make it clear that the regularity of the bundle plays an important role in the shape of the manifolds $\mathcal{W}_{\epsilon}$ but the spread of the sub-Riemannian ball depends mainly only on geometric properties of $\Delta$, in particular the non-involutivity amount.

By assumptions of the theorem, we have a $p_0 \in M$ where the condition of non-integrability is satisfied and that $\mathcal{A}_0^1(\Delta)(M)$ is equipped with a continuous exterior differential $\{V,\eta,d\eta\}$. We assume we are given an adapted coordinate system with a domain $U$ and an adapted basis which satisfies the properties listed in section \ref{sssection-fixU}. We let $\Pi_x$ denote the projection to the $x$ coordinates and $\Pi_y$ to $y$ coordinates.

\subsubsection{Part I: Construction of $\mathcal{W}_{\epsilon}$ and Its Properties}\label{ssec-partI}

Given any $|\epsilon| \leq \epsilon_0$ we can define the function $T_{\epsilon}: (-\epsilon,\epsilon)^n \rightarrow V$ by
$$
T_{\epsilon}: (t_1,\dots,t_n) \rightarrow e^{t_nX_n}\circ \dots \circ e^{t_1X_1}(0).
$$
This function is $C^1$ by condition \eqref{eq-remaininside2}. Notice also that it is 1-1 since $X_i$ are uniquely integrable and since due to their form, an integral curve of $X_i$ can intersect an integral curve of $X_j$ for $i\neq j$ only once. Therefore the image of $T$ which we denote as $\mathcal{W}_{\epsilon}$ is a $C^1$ surface. Moreover, every point on it is obviously accessible. Finally it can be given as a graph over $(x_1,\dots,x_n)$, in fact due to the form of the vector-fields $T_{\epsilon}(x_1,\dots,x_n) = (x_1,\dots,x_n, a(x_1,\dots,x_n))$ for some $C^1$ function $a$. In particular note that if
$(x,y)= T_{\epsilon}(t_1,\dots,t_n)$ then $|x| = |t|$. By the condition
\eqref{eq-remaininside2}, $\mathcal{W}_{\epsilon} \subset U$. Two important properties that we will use often are given in the next lemma:
\begin{lem}\label{lem-propertiesofW} Let $\mathcal{W}_{\epsilon}$ be as defined above with $\epsilon \leq \epsilon_0$. Then,

\begin{itemize}

\item For any $p=(x,y) \in U$ with $|x^i| \leq \epsilon$, there exists a unique $q \in \mathcal{W}_{\epsilon}$ with the same $x$ coordinates as $p$.

\item For any $q=(x,y) \in \mathcal{W}_{\epsilon}$ one has that $|y| \leq |x|\tilde{C}\omega_{2|x|}$.

\end{itemize}

\end{lem}

\begin{proof} For the first item simply note that $T_{\epsilon}(t_1,\dots,t_n) = (t_1,\dots,t_n, a(t_1,$ $\dots,t_n))$ which is a map defined for $|t_i| \leq \epsilon$.
So if $|x^i| \leq \epsilon$ then we have $T_{\epsilon}(x^1,\dots,x^n)=(x,a(x))$. Uniqueness then follows from the observation that it is a graph over the $x$ coordinates.

For the second let $q=T(t_1,\dots,t_n)$ and $\tilde{X}_i = sign(t_i)X_i$. Then, $q$ lies at the end of a piecewise $C^1$ curve $\tau :[0,|t|] \rightarrow U$ which is a concatenation of the curves $s\rightarrow e^{s \tilde{X}_i}(e^{t_{i-1} X_{i-1}}(\dots e^{t_1 X_1}(0)))$ for $0 \leq s \leq |t_i|$. So for a.e $s \leq |t|$, $\dot{\tau}(s) = \tilde{X}_{\ell(s)}(\tau(s))$ for some piecewise constant function $\ell(t)$.  Consider the vector-field $Z(s,p) = sign(t_{\ell(s)})\partial_{\ell(s)}$ with an integral curve $v(s)$ starting at $0$ which is a curve in the $x$ plane. We have $\Pi_x(\tau(s)) = \Pi_x(v(s))$ and $\Pi_y(v(s)) = 0$. So
$$
|\tau(s)-v(s)| = |\Pi_y(\tau(s)-v(s))|=|\Pi_y(\tau(s))| .
$$
But since $|t| =|x|$,
$$
|\tau(s)-v(s)| \leq \int_0^{s}| X_{\ell(w)}(\tau(w))-  \partial_{\ell(w)}(v(w))|dw \leq |x|\sup_{w \leq t}|X_{\ell(w)}(\tau(w))-  \partial_{\ell(w)}(v(w))|.
$$
Then, using $\partial_{\ell(w)}(v(w)) = X_{\ell(w)}(0)$ and $|\tau(w)| \leq |\tau| \leq 2|x|$ (by condition \eqref{eq-normX}), one has that
$$
|\tau(s)-v(s)| \leq |x| \tilde{C}\omega_{2|x|},
$$
which implies $|y| \leq |x| \tilde{C}\omega_{2|x|}$.

\end{proof}

\subsubsection{PartII: Description of the Spread of $B_{\Delta}(0,\epsilon)$ around $\mathcal{W}_{\epsilon_0}$}\label{ssection-spread}

Define for $p=(x,y)$ with $|x| \leq \epsilon_0$, $d_y(p,\mathcal{W}_{\epsilon_0})$ to be the distance of
$p$ to the point $q$ on $\mathcal{W}_{\epsilon_0}$ with the same $x$ coordinate as $p$ (by Lemma \ref{lem-propertiesofW}, there is a unique such point). For $\epsilon \leq \epsilon_0$ we can define the box around $\mathcal{W}_{\epsilon_0}$ (using the given smooth adapted coordinates) as follows:
$$
\mathcal{BW}_{\epsilon_0}(K_1, \epsilon) = \{ p=(x,y) \in U, \  |x| \leq \epsilon \ \ s.t \ \ d_y(p,\mathcal{W}_{\epsilon_0}) \leq K_1\epsilon^2\}.
$$

Then, the lemma we will prove in this section using Proposition \ref{prop-stokessubriemann} is the following (see figure \ref{fig-ballbox}):
\begin{lem}\label{lem-distancefromW}For $\mathcal{BW}_{\epsilon_0}(K, \epsilon)$ as defined above and $K_1,K_2$ constants given in Theorem \ref{thm-ballbox}, we have
$$
\mathcal{BW}_{\epsilon_0}(K_1, \frac{\epsilon}{4d_g}) \subset B_{\Delta}(0,\epsilon) \subset  \mathcal{BW}_{\epsilon_0}(K_2, 2nd_g\epsilon).
$$
\end{lem}

\begin{figure}
  \centering
  \includegraphics[width=200px]{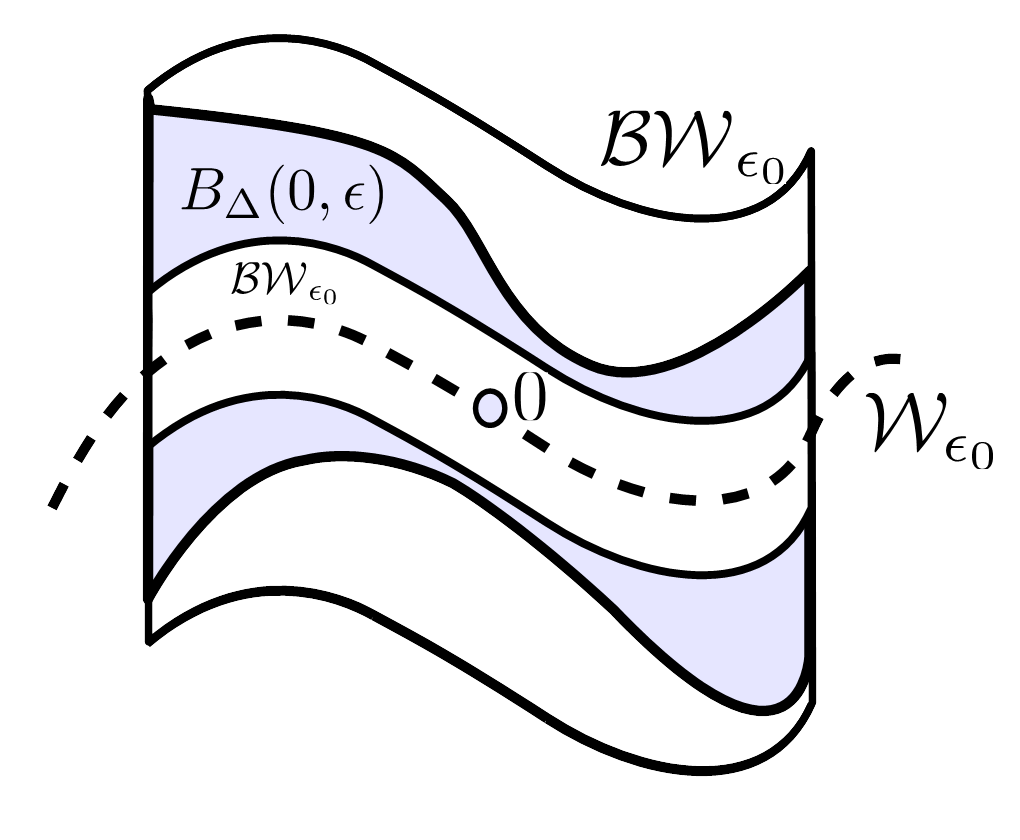}
  \caption{Pictorial representation of emma \ref{lem-distancefromW}.}\label{fig-ballbox}
\end{figure}

\begin{proof} [Proof of $\mathcal{BW}_{\epsilon_0}(K_1, \frac{1}{4d_g}\epsilon) \subset B_{\Delta}(0,\epsilon)$]

Given $\epsilon < \frac{\epsilon_0}{2nd_g}$, to prove the inclusion we need to show that if $p=(x,y)$ is such that $|x| \leq \frac{1}{4d_g}\epsilon$ and $d(p,\mathcal{W}_{\epsilon_0}) \leq K_1\frac{1}{16d^2_g}\epsilon^2$, then $d_{\Delta}(x,0) \leq \epsilon$. For this
it is enough to find a path $\psi$ between $p$ and $0$ which is admissible and
$\ell(\psi) \leq \epsilon$, since $d_{\Delta}(x,0) \leq \ell(\psi)$ where here $\ell(\cdot)$ is the length defined with respect to the metric $g$ defined on $\Delta$. This path $\psi$ will be written as the composition of two paths $\psi = \gamma \circ \tau$ where $\tau$ is a path that allows to travel along $x$ direction and $\gamma$ is a path which will allow us to travel in the $y$ direction and for which we will apply the Proposition \ref{prop-stokessubriemann}.

First we construct $\tau$. Let $q_1$ be the point in $\mathcal{W}_{\epsilon_0}$ that has the same $x$ coordinates as $p$ (since $|x| \leq \frac{\epsilon}{4d_g} < \epsilon_0$ such a point exists). Let $\tau$ be the curve in $\mathcal{W}_{\epsilon_0}$ described in Lemma \ref{lem-propertiesofW} that connects $0$ to $q_1$.
It is curve defined from $[0,|x|]$ to $U$. We have that by condition \eqref{eq-normX}
\begin{equation}\label{eq-taulength}
|\tau| \leq 2|x| \leq \frac{\epsilon}{2d_g}
\end{equation}

Using equation \eqref{eq-taulength}, we have that $\ell(\tau) \leq d_g |\tau| \leq \frac{\epsilon}{2}$.
Now we need to build the other admissible curve $\gamma$ that starts at $q_1$ and ends at $p$ and such that $\ell(\gamma) \leq \frac{\epsilon}{2}$. If we build this curve then we are done
since both $\tau$ and $\gamma$ are admissible and $\ell(\gamma \circ \tau) \leq \epsilon$.

The amount that we need to travel in the $\partial_y$ direction is given by the condition $d(p,\mathcal{W}_{\epsilon_0}) \leq K_1\frac{1}{16d^2_g}\epsilon^2$
which tells us that
\begin{equation}\label{eq-neededyrise}
|q_1 - p| \leq K_1\frac{1}{16d^2_g}\epsilon^2.
\end{equation}

Now we start building $\gamma$. The idea is similar to the one employed in \cite{Arn89} (see Section 36 of Chapter 7 and Appendix 4) which defines exterior differential of 1-forms in terms of loops tangent to the bundles that they annihilate. For any $\tilde{\epsilon} \leq \epsilon_0$, consider the curves defined on $[0,\tilde{\epsilon}]$ (the indexing of points below are chosen so that they are in direct alignment
with Proposition \ref{prop-stokessubriemann}):
$$
\kappa_1(s)=e^{sX_i}(q_1)  \quad \quad
\kappa_2(s)= e^{sX_j}( \kappa_1(\tilde{\epsilon})) \quad \quad \kappa_2(\tilde{\epsilon})= q,
$$
$$
\kappa_3(s)=e^{-sX_i}(q)  \quad \quad
\kappa_4(s)= e^{-sX_j}(\kappa_3(\tilde{\epsilon})) \quad \quad \kappa_4(\tilde{\epsilon})= q_2.
$$
These $X_i$ and $X_j$ are the vector-fields that satisfy $|d\eta(X_i,X_j)|_{\inf}>0$.
We let $\gamma(s;\tilde{\epsilon},q_1)$ be the parametrization for the curve obtained as the concatenation $\kappa_4\circ \kappa_3 \circ \kappa_2 \circ \kappa_1$ so that  $\dot{\gamma}(s)=X_{\ell(s)}(\gamma(s;\tilde{\epsilon},q_1))$ for a.e $s$ with $\ell(s)=i,j$. Also $q_2=\gamma(4\tilde{\epsilon}; \tilde{\epsilon}, q_0)$ becomes the end point of this curve (see figure \ref{fig-part1}). This curve is defined on $[0,4\tilde{\epsilon}]$ to $U$. Moreover by the condition \eqref{eq-remaininside2}, we have that
$\gamma(s;\tilde{\epsilon},q_0)$ is continuous with respect to $\tilde{\epsilon}$ since it is equal to $\gamma(s;\tilde{\epsilon},\tau(\tilde{\epsilon}))$ which is just a composition of $n+4$ integral curves of some $\pm X_{\ell}$ with integration times less than $\tilde{\epsilon}_0 \leq \epsilon_0$. Denoting the image of this curve as $\gamma_{\tilde{\epsilon}}$ we have $\ell(\gamma_{\tilde{\epsilon}}) \leq d_g|\gamma_{\tilde{\epsilon}}| \leq 8d_g\tilde{\epsilon}$.
Therefore the following lemma is sufficient to get $\gamma$.

\begin{figure}
  \centering
  \includegraphics[width=200px]{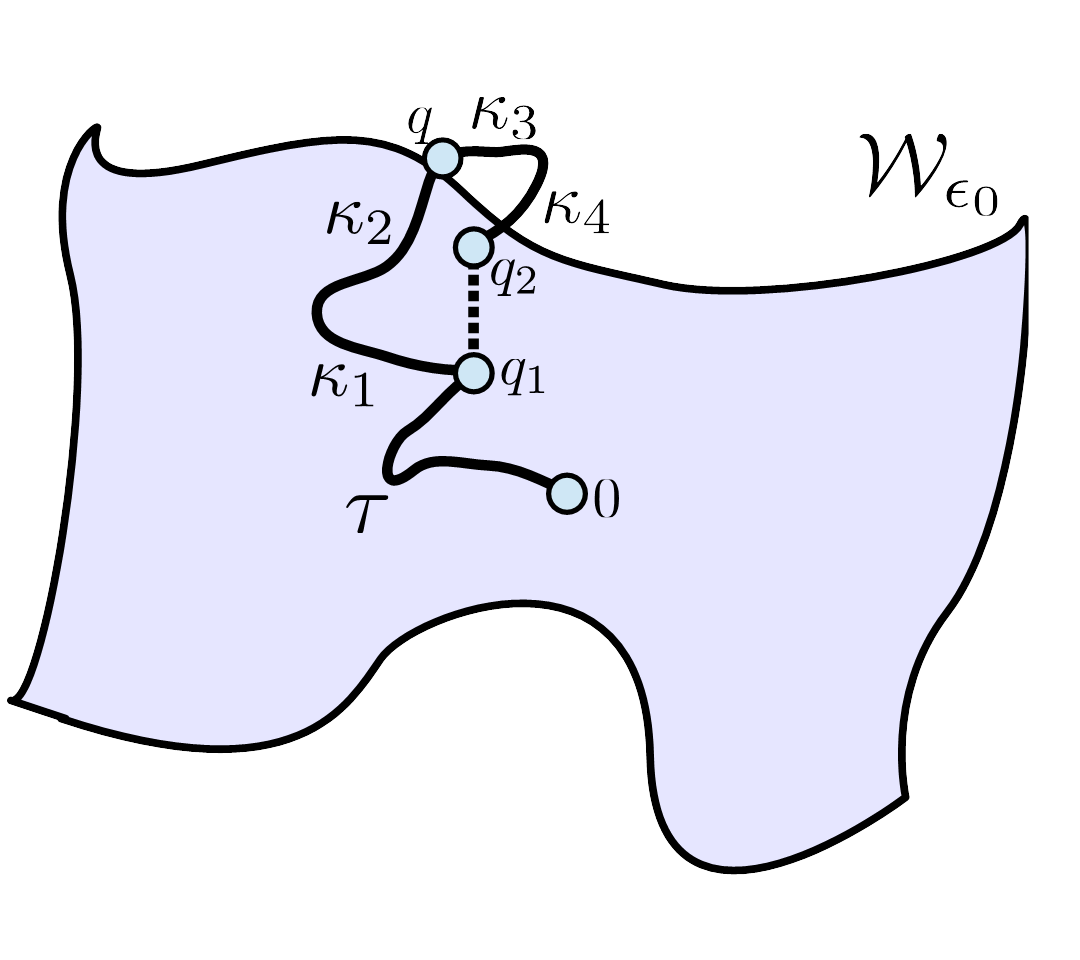}
  \caption{Trying to get $p=q_2$.}\label{fig-part1}
\end{figure}

\begin{lem}\label{lem-final}There exists some $\tilde{\epsilon}$ such that $\gamma(4\tilde{\epsilon};\tilde{\epsilon},q_1) = p$ with $\tilde{\epsilon} \leq \frac{\epsilon}{16d_g}$ \end{lem}

If we can show this we are done since then setting $\gamma=\gamma_{\tilde{\epsilon}}$,
$\ell(\gamma) \leq d_g|\gamma| \leq 8d_g\tilde{\epsilon} \leq \frac{1}{2}\epsilon$.

\begin{proof}

To prove the lemma it is sufficient to show the following:

\begin{itemize}

\item There exists some $\bar{\epsilon}$ with $\bar{\epsilon} \leq \frac{\epsilon}{16d_g}$ such that  $|\gamma(4\bar{\epsilon};\bar{\epsilon},q_1)-q_1| \geq  K_1\frac{1}{16d^2_g}\epsilon^2.$

\item  By reverting either $X_i$ or $X_{j}$ we can go in the opposite direction along the $\partial_y$.
\end{itemize}

The last item allows us to travel in both directions (with respect to $\partial_y$) while the second item guarantees that we can travel more than $|p-q_1|$. Then, since  $\gamma(4{\epsilon};{\epsilon},q_1)$ continuous with respect to ${\epsilon}$ and it always lies on the one dimensional $\partial_y$ axis passing through $p$ and $q_1$ and satisfies $\gamma(0;0,q_1)=q_1$, there exists some $\tilde{\epsilon} \leq \bar{\epsilon}$ (so that automatically  $\tilde{\epsilon} \leq \frac{\epsilon}{16d_g}$ )  for which $q_2=\gamma(4\tilde{\epsilon};\tilde{\epsilon},q_1)=p$ and we are done with the lemma. So we will now prove these items.

To apply Proposition \ref{prop-stokessubriemann}, set
$$
 \gamma_1 = \kappa_1^{-1} \circ \kappa_2^{-1}, \quad \quad \gamma_2 = \kappa_3 \circ \kappa_4,
$$
so that $\gamma_1(0)=\gamma_2(0) = q$ and $\gamma_1(\varepsilon_1)=q_1$, $\gamma_2(\varepsilon_2) = q_2$ which have the same $x$ coordinates due to the form of the vector-fields $X_i,X_j$.  We also have (in the terminology of Proposition  \ref{prop-stokessubriemann}) $\varepsilon \leq 2\epsilon$, $\ell \leq 4\epsilon$. By conditions \eqref{eq-normX}, \eqref{eq-normX2}, \eqref{eq-normX3} we have
$$
\xi \leq (2n)^{n+2}\tilde{C}\omega_{8\epsilon}.
$$
Then, by condition \eqref{eq-remaininside1} we have that
$B(2\ell ,q_0) \subset U$. Now it remains to build a surface $P$ at the point $q$ so that $P \subset \bar{X}_{q} \cap U$. Note that projection of $\gamma_i$ along $\partial_y$ to $\bar{X}_{q}$ is simply the boundary of a parallelogram $P$ whose sides are given by the vectors $X_i(q)$,$X_j(q)$ and have length less than $\ell \leq 4\epsilon$. Therefore given any point $z \in P$, $|z| \leq |z-q| + |q_1 - q| + |q_1|$. But
$|q_1| \leq 2|x| \leq \frac{\epsilon}{2}$, $|q_1 - q| \leq \ell \leq 4\epsilon$, and $|z-q| \leq 2\ell \leq 8\epsilon$. So
$|p| \leq 13\epsilon$ which by condition \eqref{eq-remaininside1} gives us that $P \subset U$. Therefore all the conditions required for application
of \ref{prop-stokessubriemann} is satisfied. So

\[
\frac{1}{|\eta(\partial_y)|_{\infty}}\bigg(\bigg|\int_P d\eta \bigg| - c\bigg) \leq |\gamma_1(\epsilon_1) - \gamma_2(\epsilon_2)|=|q_2-q_1|,
\]
and
\[
sign\bigg(\int_{\beta}\eta \bigg) = sign\bigg(\int_P d\eta + c\bigg),
\]
where also by condition \ref{eq-estimate1},
\begin{equation}\label{eq-cnorm}
|c| \leq 4 \ell \varepsilon \xi |d\eta|_{\infty} \leq  (2n)^{n+7} \epsilon^2\tilde{C}\omega_{8\epsilon}|d\eta|_{\infty}   \leq \frac{1}{4}|d\eta(X_i,X_j)|_{\inf}\epsilon^2.
\end{equation}

To apply these result we need to calculate $\int_P d\eta$ which is the content of the next lemma:
\begin{lem}For $P$ as defined above,
\begin{equation}\label{eq-parallelogramintegral}
\bigg|\int_{P}d\eta \bigg|  \geq \frac{3}{4}\epsilon^2 |d\eta(X_i,X_j)|_{\inf}.
\end{equation}
\begin{equation}\label{eq-parallelogramintegral2}
sign\bigg(\int_{P}d\eta \bigg) =sign(d\eta_{q}(X_i(q),X_j(q))
\end{equation}

\end{lem}

\begin{proof}
Note that $P$ is a parallelogram based at $q$ and always tangent to $X_i(q)$ and $X_j(q)$. We will denote by $P(s_1,s_2)$ a parametrization for $P$ such that $\frac{\partial P}{\partial s_1}(s_1,s_2)=X_i(q)$ and $\frac{\partial P}{\partial s_2}(s_1,s_2)=X_j(q)$ for all $0 \leq s_1,s_2 \leq \epsilon$. Then, we have in this parametrization,
$$
\int_{P}d\eta = \int_{0}^{\epsilon}ds_1 \int_{0}^{\epsilon}ds_2 \ \  d\eta_{P(s_1,s_2)}(X_i(q),X_j(q)).
$$
But
\[
\begin{aligned}
d\eta_p(X_i(q),X_j(q)) &= d\eta_p(X_i(p),X_j(p)) \\
&+ d\eta_p(X_i(q)-X_i(p),X_j(q)) + d\eta_p(X_i(q),X_j(q)-X_j(p))\\
&+d\eta_p(X_i(q)-X_i(p),X_j(q)-X_j(p)).\\
\end{aligned}
\]
Note that $|X_{\ell}(q)-X_{\ell}(p)| \leq \tilde{C}\omega_{|q-p|}$ and for $p \in P$,  $|q-p| \leq 13\epsilon$.
 So
\[
\begin{aligned}
|d\eta_p(X_i(q)&-X_i(p),X_j(q)) + d\eta_p(X_i(q),X_j(q)-X_j(p))\\
&+d\eta_p(X_i(q)-X_i(p),X_j(q)-X_j(p))| \leq   |d\eta|_{\infty}\tilde{C}\omega_{13\epsilon}(4 + \tilde{C}\omega_{13\epsilon}),\\
\end{aligned}
\]
and therefore
\begin{equation}\label{eq-parallelogramestimate}
d\eta_p(X_i(q),X_j(q)) \geq d\eta_p(X_i(p),X_j(p)) - |d\eta|_{\infty}\tilde{C}\omega_{13\epsilon}(4 + \tilde{C}\omega_{13\epsilon}),
\end{equation}
and
\begin{equation}\label{eq-parallelogramestimate2}
d\eta_p(X_i(q),X_j(q)) \leq d\eta_p(X_i(p),X_j(p)) + |d\eta|_{\infty}\tilde{C}\omega_{13\epsilon}(4 + \tilde{C}\omega_{13\epsilon}).
\end{equation}
Then, using the condition \eqref{eq-estimate1} we have that
\begin{equation}\label{eq-bounds}
\frac{3}{4}d\eta_p(X_i(p),X_j(p))  \leq d\eta_p(X_i(q),X_j(q)) \leq \frac{5}{4}d\eta_p(X_i(p),X_j(p)). \end{equation}
So for all $p \in P$, $sign(d\eta_p(X_i(q),X_j(q)) )=sign(d\eta_p(X_i(p),X_j(p)) )$.
But in $U$, $d\eta_p(X_i(p),X_j(p))$ is never $0$ and therefore never changes sign so this proves equation \eqref{eq-parallelogramintegral2}.
This also means that $d\eta_p(X_i(q),X_j(q))$ never changes sign and using equation \eqref{eq-bounds} we get,
$$
\bigg|\int_{P}d\eta \bigg| = \int_{0}^{\epsilon}ds_1 \int_{0}^{\epsilon}ds_\frac{3}{4} |d\eta_{P(s_1,s_2)}(X_i(q),X_j(q))| \geq \epsilon^2 |d\eta(X_i,X_j)|_{\inf}.
$$
\end{proof}

Now we have that $sign(\int_{\beta}\eta) = sign(\int_P d\eta + c)$ where by the equation \eqref{eq-parallelogramintegral} in the previous lemma and equation \eqref{eq-cnorm}
$$
|c| \leq \frac{1}{4}|d\eta(X_i,X_j)|_{\inf}\epsilon^2\leq \frac{1}{3}\bigg |\int_{P}d\eta \bigg |.
$$
So $sign(\int_P d\eta + c)=sign(\int_Pd\eta) = sign(d\eta_{q}(X_i,X_j))$ (which can be reverted by changing the direction of, say $X_i$).
Moreover
$$
|\alpha(4\epsilon;\epsilon,q_1)-q_1| \geq \frac{1}{|\eta(\partial_y)|_{\infty}}\bigg(\bigg|\int_{P}d\eta \bigg|  - |c|\bigg) \geq \frac{8}{12}\epsilon^2 \frac{|d\eta(X_i,X_j)|_{\inf}}{|\eta(\partial_y)|_{\infty}}.
$$
Since by condition \ref{eq-normX2} we have $|d\eta(X_i,X_j)|_{\inf} \geq \frac{1}{1.75}m(d\eta|_{\Delta})_{\inf}$, replacing now $\epsilon$ with $\bar{\epsilon}$
\begin{equation}\label{eq-endpoint}
|\alpha(4\bar{\epsilon};\bar{\epsilon},q_1)-q_1| \geq  \frac{8}{21}\bar{\epsilon}^2 \frac{m(d\eta|_{\Delta})_{\inf}}{|\eta(\partial_y)|_{\infty}}.
\end{equation}
We require this quantity to be bigger than $ K_1\frac{1}{16d^2_g}\epsilon^2$ $\geq  K_1\frac{1}{16d^2_g}256 d_g^2\bar{\epsilon}^2$=$ 16K_1\bar{\epsilon}^2$
So we need to satisfy
$$
\frac{1}{42}\bar{\epsilon}^2 \frac{m(d\eta|_{\Delta})_{\inf}}{|\eta(\partial_y)|_{\infty}} \geq K_1\bar{\epsilon}^2.
$$
This is satisfied since
$$
K_1 =  \frac{1}{42} \frac{m(d\eta|_{\Delta})_{\inf}}{|\eta(\partial_y)|_{\infty}}
$$

\end{proof}
This finishes the proof of the inclusion $\mathcal{BW}_{\epsilon_0}(K_1,\frac{\epsilon}{4d_g}) \subset B_{\Delta}(0,\epsilon) $.

\end{proof}

\begin{proof} [Proof of  $B_{\Delta}(0,\epsilon) \subset  \mathcal{BW}_{\epsilon_0}(K_2, 2nd_g\epsilon)$]

To prove this inclusion we need to prove that if $p=(x,y) \in U$ such that
$d_{\Delta}(p,0) \leq \epsilon$, then $|x| \leq 2nd_g\epsilon$ and $d(p,\mathcal{W}_{\epsilon_0}) \leq K_24n^2d_g^2\epsilon^2$

$d_{\Delta}(p,0) \leq \epsilon$ means that there exists an admissible curve $\gamma_1$ that connects $0$ to $p$ such that $\ell(\gamma_1) \leq 2\epsilon$. Then
$|\gamma_1| \leq d_g\ell(\gamma_1) \leq 2d_g\epsilon$. Since for all $i$ $|x^i| \leq |\gamma_1|$ we get
$|x| \leq 2nd_g\epsilon $ trivially. So it remains to estimate $d(p,\mathcal{W}_{\epsilon_0})$.

The rest of the proof follows the method given in \cite{Gro96} by Gromov. The only essential difference is that the term $|d\eta|_{\infty}$ appearing there is replaced by $|d\eta|_{\Delta}|_{\infty}$ which
will be thanks to Proposition \ref{prop-stokessubriemann}.

We first state a lemma due to Gromov, (\cite{Gro83}, Sublemma $3.4$.B, see also Corollary $2.3$ in \cite{Sim10}) specialized to lower dimension:
\begin{lem}\label{lem-Gromov}For every compact Riemannian manifold $S$, there exists constants $\delta_S,c_S>0$ such that for any 1-cycle $\gamma$ in $S$ with length less than $\delta_S$, there exists a 2-chain $\Gamma$ in $S$, bounded by $\gamma$ such that
$$
|\Gamma| \leq c_S |\gamma|^2,
$$
and $\Gamma$ is contained in the $\varrho$ neighbourhood of $\gamma$, where $\varrho=c_S|\gamma|$.
\end{lem}

\begin{rem}\label{rem-Gromov}
Here $|\cdot|$ denotes the norm of the chosen metric on $S$ and the constants $\delta_S,c_S$ depend on this metric. We will always apply this result to closures of precompact open submanifolds of $n$ dimensional Euclidean space which are given by
$S = U\cap \text{span}\{\frac{\partial}{\partial x^i}\}|_{q}$ for some $q \in U$. They will have the usual Euclidean inner product. These spaces are affine translates of each other (since $U$ is a box) all equipped with Euclidean metrics for which these affine translations are isometries. Therefore it is clear that whatever constant $c_S$ and $\delta_S$ works for one, also works for the other. We denote these constants simply by $\delta$ and $c$.
Actually note that the proof of this lemma (\cite{Gro83}, Sublemma $3.4$.B), starts by first finding an embedding $\phi$ of $S$ into some $\mathbb{R}^N$. Then, a 2-chain with the required properties is constructed inside $\mathbb{R}^N$ and normally projected back to $S$. So the constant $c_S$ and $\delta_S$ depend only on the embedding and $N$. In our case since we are working with affine translates of open subsets of $\mathbb{R}^n$, the constants $c$ and $\delta$ depend only on the dimension $n$ since the embedding becomes an isometry and normal projection is not required.
\end{rem}

Now since $2nd_g\epsilon \leq \epsilon_0$, there exists a point $q$ on $\mathcal{W}_{\epsilon_0}$ which has the same $x$ coordinates as $p$. This means
 $d(p,\mathcal{W}_{\epsilon_0})=|q-p|$ and so we need to show that $|q-p| \leq  K_24n^2d_g^2\epsilon^2$.
We can connect $p$ to $0$ by first going from $0$ along $\mathcal{W}_{\epsilon_0}$ with an admissible path
$\gamma_2$ to $q$ and then going in the $\partial_y$ direction with a length parametrized segment $\beta$ (see figure
\ref{fig-part2}). Then, $\gamma_1, \gamma_2, \beta$ forms a 1-cycle.

\begin{figure}
  \centering
  \includegraphics[width=200px]{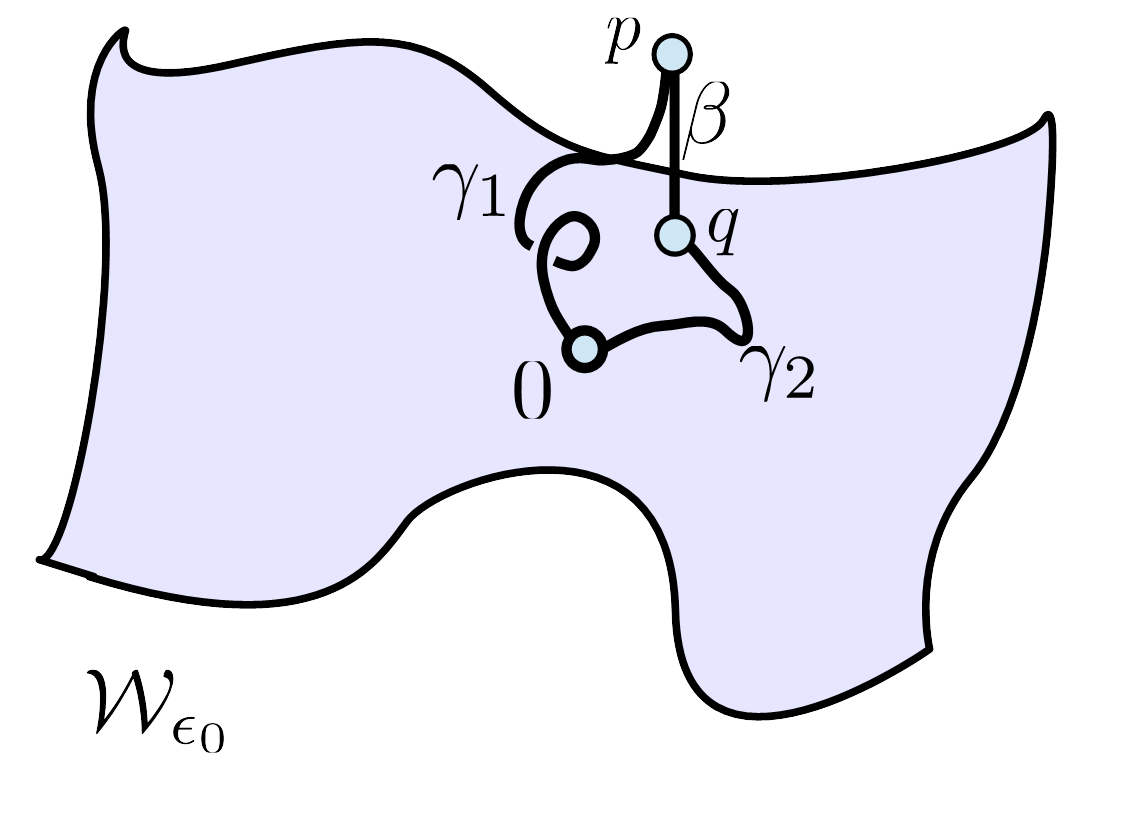}
  \caption{The Curve $\gamma_2\circ \beta \circ \gamma^{-1}_1$.}\label{fig-part2}
\end{figure}

Now to apply Proposition \ref{prop-stokessubriemann} we have
$\gamma_1(0)=\gamma_2(0)=0$ and $\gamma_1(\epsilon_1)=p$, $\gamma_2(\epsilon_2)=q$ which have the same $x$ coordinates.
Then
$$
\ell_1 \leq 2d_g\epsilon \leq \epsilon_0, \quad \quad \ell_2 \leq 2|x| \leq 4nd_g\epsilon \leq 2\epsilon_0, \quad \quad \ell \leq 4nd_g\epsilon \leq 2\epsilon_0,
$$
$$
\varepsilon_1 \leq 2d_g\epsilon \leq \epsilon_0, \quad \quad \varepsilon_2 \leq |x|  \leq 2nd_g\epsilon \leq \epsilon_0, \quad \quad \varepsilon \leq \epsilon_0.
$$
Also by conditions \eqref{eq-normX} and \eqref{eq-normX3} we have that
$$
\xi \leq (2n)^{n+2}\tilde{C}\omega_{4\epsilon_0}.
$$
Therefore by condition \eqref{eq-remaininside1}, $B(0, 2\ell) \subset U$.
The projection of $\gamma_1$ and $\gamma_2$ along $\partial_y$ to $\bar{X}_{0}$ formes a 1-cycle $\alpha=\alpha^{-1}_1 \circ \alpha_2$ which has length less than $2d_g\epsilon +4nd_g\epsilon \leq 3\epsilon_0$.
By the condition \eqref{eq-estimate3} this is less than $\delta$ so by Lemma \ref{lem-Gromov}, $\alpha$ bounds a 2-cycle $P \subset \bar{X}_{0}$ such that $|P| \leq c(2+4n)^2d^2_g\epsilon^2$
and which is in the $c(2+4n)d_g\epsilon$ neighbourhood of $\alpha$. We have that for any $z \in P$
\[
\begin{aligned}
d(z,0) &\leq |\alpha| + c(2+4n)d_g \\
& \leq  (c+1)(2+6n)d_g\epsilon,\\
\end{aligned}
\]
so by condition \eqref{eq-remaininside1} we have that $P \subset U \cap \bar{X}_{0}$. Then, the requirements of the Proposition \ref{prop-stokessubriemann} are satisfied. Since $4\ell\varepsilon\xi \leq (2n)^{n+7}d_g^2\tilde{C}\omega_{8\epsilon_0}\epsilon^2$, we have
\begin{equation}\label{eq-lastestimate}
|p-q| \leq \frac{1}{|\eta(\partial_y)|_{\inf}}\bigg(\bigg|\int_P d\eta \bigg| + 4\ell\varepsilon\xi |d\eta|_{\infty}\bigg) \leq   \frac{1}{|\eta(\partial_y)|_{\inf}}\bigg( \bigg |\int_P d\eta \bigg| +  (2n)^{n+7}d_g^2\tilde{C}\omega_{8\epsilon_0}\epsilon^2|d\eta|_{\infty}\bigg).
\end{equation}
We need to estimate $|\int_P d\eta|$.
\begin{lem}For $P$ as defined above,
$|\int_P d\eta| \leq 4n^2c(2+4n)^2d^2_g\epsilon^2   |d\eta|_{\Delta}|_{\infty}.$
\end{lem}
\begin{proof} Since $P$ is inside $\bar{X}_{0}$, it is everywhere tangent to $X_i(0)=\partial_i$ so
$|\int_P d\eta| \leq |P| |d\eta|_{\Delta_0}|_{\infty}$. But for all $p\in U$,

$$
|d\eta_p|_{\Delta_0}|_{\infty} \leq n^2 \sup_{\ell,k=1,\dots,n}|d\eta_p(\partial_{\ell},\partial_k)|.
$$
Since $X_{\ell}(0) = \partial_{\ell}$, then for any $z \in P$ we have as in equation \eqref{eq-parallelogramestimate2},
$$
|d\eta_z(\partial_{\ell},\partial_k)| \leq |d\eta_z(X_{k}(z),X_l(z))| + |d\eta|_{\infty}\tilde{C}\omega_{|z|}(4 + \tilde{C}\omega_{|z|}).
$$
Then, since $|z| \leq  (c+1)(2+6n)d_g\epsilon \leq 4(c+1)\epsilon_0$
by \eqref{eq-estimate1} we get by condition \eqref{eq-normX2}
$$
|d\eta_p(\partial_{\ell},\partial_k)| \leq 2|d\eta_p(X_{\ell}(p),X_l(p))| \leq 4 |d\eta|_{\Delta}|_{\infty},
$$
which implies
$$
|d\eta_p|_{\Delta_0}|_{\infty} \leq 8n^2  |d\eta|_{\Delta}|_{\infty}.
$$
So since $|P| \leq  c(2+4n)^2d^2_g\epsilon^2$,
$$
\bigg|\int_P d\eta \bigg| \leq  4n^2c(2+4n)^2d^2_g\epsilon^2|d\eta|_{\Delta}|_{\infty}.
$$
\end{proof}

Then this lemma and equation \eqref{eq-lastestimate} gives
$$
|p-q| \leq   \frac{1}{|\eta(\partial_y)|_{\inf}}(4n^2c(2+4n)^2d^2_g\epsilon^2   |d\eta|_{\Delta}|_{\infty} +  (2n)^{n+7}d_g^2\tilde{C}\omega_{8\epsilon_0}\epsilon^2|d\eta|_{\infty}).
$$
Condition \eqref{eq-estimate1} implies then
$$
|p-q| \leq   5n^2c(2+4n)^2d^2_g\epsilon^2  \frac{ |d\eta|_{\Delta}|_{\infty}}{|\eta(\partial_y)|_{\inf}}.
$$
So to be able to satisfy $|p-q| \leq 4K_2n^2d_g^2\epsilon^2$, it is sufficient to satisfy
$$
9n^2c(2+4n)^2d^2_g\epsilon^2  \frac{ |d\eta|_{\Delta}|_{\infty}}{|\eta(\partial_y)|_{\inf}} \leq 4K_2d_g^2\epsilon^2,
$$
which is satisfied with
$$
K_2 = 42(1+2n)^2c \frac{ |d\eta|_{\Delta}|_{\infty}}{|\eta(\partial_y)|_{\inf}}.
$$

\end{proof}

\subsubsection{Rest of the Proof of Theorem \ref{thm-ballbox}}

Now it is easy to prove the rest using Lemmas \ref{lem-propertiesofW} and \ref{lem-distancefromW}.
The latter one says that
$$
\mathcal{BW}_{\epsilon_0}(K_1, \frac{\epsilon}{4d_g}) \subset B_{\Delta}(0,\epsilon) \subset  \mathcal{BW}_{\epsilon_0}( K_2, 2nd_g\epsilon),
$$
while the former one says that for $q=(x,y) \in \mathcal{W}_{\epsilon_0}$ we have that
$$
|y| \leq |x|\tilde{C}\omega_{2|x|}.
$$

First we prove that $\mathcal{BW}_{\epsilon_0}( K_2, 2d_g\epsilon) \subset  H_2^{\omega}(0,K_2,2d_g\epsilon)$. If $(x,y) \in   \mathcal{BW}_{\epsilon_0}($ $K_2, 2nd_g\epsilon)$, then $|x| \leq 2nd_g\epsilon$ and $d(p,\mathcal{W}_{\epsilon_0}) \leq 4K_2n^2d_g^2\epsilon^2$.
This means there exists $q=(x,z) \in \mathcal{W}_{\epsilon_0}$ such that $|z-y| \leq 4K_2n^2d_g^2\epsilon^2$. But we know that $|z| \leq |x|\tilde{C}\omega_{2|x|}$.
So $|y| \leq |z| + |z-y| \leq 4K_2n^2d_g^2\epsilon^2 + |x|\tilde{C}\omega_{2|x|}$ which means that $(x,y) \in H_2^{\omega}(0,K_2,2nd_g\epsilon)$.

Now we prove $D_2^{\omega}(0,\frac{1}{K_1}, \frac{\epsilon}{4d_g}) \subset \mathcal{BW}_{\epsilon_0}(K_1, \frac{\epsilon}{4d_g})$.
If $p=(x,y) \in D_2^{\omega}(0,\frac{1}{K_1},$ $\frac{\epsilon}{4d_g})$, then $|x| + \sqrt{\frac{1}{K_1}(|x|\omega_{2|x|} + |y|)} \leq \frac{\epsilon}{4d_g}$. So we have
$|x| \leq  \frac{\epsilon}{4d_g}$ and $(|x|\omega_{2|x|} + |y|) \leq K_1\frac{\epsilon^2}{16d_g^2}$. Let  $q=(x,z) \in \mathcal{W}_{\epsilon_0}$ with
$|z| \leq |x|\tilde{C}\omega_{2|x|}$.  Therefore $|y| + |z| \leq  K_1\frac{\epsilon^2}{16d_g^2}$. So $d(p,\mathcal{W}_{\epsilon_0})= |p-q|=|y-z| \leq |y| + |z| \leq  K_1\frac{\epsilon^2}{16d_g^2}$.
This implies that $(x,y) \in  \mathcal{BW}_{\epsilon_0}(K_1, \frac{\epsilon}{4d_g})$.

Now finally we prove the statement about the existence of $C^1$ adapted coordinates. We will build the $C^1$ coordinate system using the manifold $\mathcal{W}_{\epsilon_0}$. Note that for each $\epsilon$, $\mathcal{W}_{\epsilon}$ are $C^1$ surfaces given as the image of the map $T_{\epsilon}(t_1,\dots,t_n)$ with $|t_i| \leq \epsilon$ and for $\epsilon_1 < \epsilon_2$, $\mathcal{W}_{\epsilon_1} \subset \mathcal{W}_{\epsilon_2}$. Define the transformation $\phi: V \rightarrow U$ on some appropriately sized domain $V\subset U$ such that
\begin{equation}\label{eq-c1adapted}
\phi(x,y) = (x,y-T_{\epsilon_0}(x)).
\end{equation}
Then, it is clear that this map is a $C^1$ diffeomorphism onto its image and takes each $\mathcal{W}_{\epsilon} \cap V$ (for $\epsilon \leq \epsilon_0$) to the $x$ plane. It also maps $X_{\ell}$ restricted to $\mathcal{W}_{\epsilon}$ to $\partial_{\ell}$ on the $x$ plane. This is again an adapted coordinate system. In particular in this adapted coordinate system $(x,y) \in   \mathcal{BW}_{\epsilon_0}( K, \epsilon)$ simply implies $|x| \leq \epsilon$ and $|y| \leq K\epsilon^2$. So we get  $\mathcal{BW}_{\epsilon_0}(K_1, \frac{\epsilon}{4d_g}) = \mathcal{B}(0,K_1, \frac{\epsilon}{4d_g})$ and $\mathcal{BW}_{\epsilon_0}( K_2, 2nd_g\epsilon)=\mathcal{B}_2(0,K_2,2nd_g\epsilon)$, which finishes the proof.

\section{Applications}\label{sec-continuousexteriordifferential}

In this section we first present interesting properties of the ``continuous exterior differential" object and give some examples of bundles
that satisfy the requirements of our main theorems. We also discuss the relations between our theorems and the works in \cite{MonMor12} and \cite{Kar11}. In the second part we present a tentative application to dynamical systems.

\subsection{Continuous Exterior Differential}

In this subsection we study some important properties of $\Omega^k_d(M)$, including of course showing that there are some non-integrable examples inside $\Omega^1_d(M) \setminus \Omega^1_1(M)$, so that we cover some examples that were not covered by $C^1$ sub-Riemannian geometry. We start with an alternative characterization that already exists in \cite{Har64}:
\begin{prop}\label{prop-char} A differential n-form $\eta$ has a continuous exterior differential if and only if there exists a sequence of differential n-forms $\eta^k$ such that $\eta^k$ converges in $C^0$ topology to $\eta$ and $d\eta^k$ converges to some differential $n+1$ form which becomes the exterior differential of $\eta$.\end{prop}

The sufficiency part of this proposition is quite direct since uniform convergence of the differential forms involved also means uniform convergence of the Stokes relation. The necessity can be obtained by locally mollifying the differential forms to $\eta^k=\phi_k\circ \eta$ (where $\phi_k$ are mollifiers). Then, under the integral $d\eta^k = d(\phi_k\circ \eta)= \phi^k d\eta$ (thanks to the fact that $\phi_k$ are compactly supported) so the Stokes relations convergence. Since Stokes relation holds for every surface and its boundary, it is enough to obtain that $d\eta^k$ themselves converge to $d\eta$.

We start with two examples defined on $U \subset \mathbb{R}^n$ and then we show how to paste local differential forms with continuous exterior differential to obtain global ones on manifolds.

\begin{ex} This is an example from \cite{Har64}. Let $f$ be any $C^1$ function so that $df$ is $C^0$. Then setting $\eta = df$, we have that $\eta$ has continuous exterior differential $0$. Therefore $\eta \wedge d\eta=0$ and this gives us integrable examples inside $\Omega^1_d(M) \setminus \Omega^1_1(M)$.
\end{ex}

The following proposition allows us to construct examples that both have continuous exterior differentials and are non-integrable:
\begin{prop}\label{prop-example}   Let $\eta =a(x,y,z)dy - b(x,y,z)dx - c(x,y,z)dz$ where $b$ and $c$ are continuous functions that are $C^1$ in $y$ variable, $b$ is $C^1$ in $z$ variable while $c$ is $C^1$ in $x$ variable and $a$ is a continuous function which is $C^1$ in the $x,z$ variables and $a>0$ everywhere. Assume moreover that for some $p \in \mathbb{R}^3$,
$$
(-(a_x + b_y)c + (a_z + c_y)b - (b_z-c_x)a)(p)>0.
$$
Then, $\ker(\eta)=\Delta$ satisfies the conditions of Theorem \ref{thm-ballbox} at $p$.
\end{prop}

\begin{proof} By assumptions on regularity of the functions, we can find $C^1$ functions (by mollification) $a^k$, $b^k$ and $c^k$ that converge in $C^0$ topology to $a$, $b$ and $c$ such that $b$ and $c$'s partial derivatives with respect to $y$, partial derivative of $b$ with respect to $z$ and partial derivative of $c$ with respect to $x$ also converges (to the respective derivatives of $b$ and $c$) and the partial derivative of $a^k$ with respect to $x,z$ converge to the respective partial derivatives of $a$. Then we define
$$
\eta^k = a^kdy + b^kdx + c^kdz,
$$
so that
$$
d\eta^k = a^k_xdx \wedge dy + a^k_z dz \wedge dy +b^k_y dy \wedge dx + c^k_y dy \wedge dz
+b^k_z dz \wedge dx +c^k_x dx \wedge dz.
$$
By assumption $\eta^k$ converges in $C^0$ topology to $\eta$ and $d\eta^k$ converges to
$$
d\eta = a_xdx \wedge dy + a_z dz \wedge dy +b_y dy \wedge dx + c_y dy \wedge dz
+b_z dz \wedge dx +c_x dx \wedge dz.
$$
Then we have that
$$
\eta \wedge d\eta(p) =(-(a_x + b_y)c + (a_z + c_y)b - (b_z-c_x)a)(p) dx^1 \wedge dx^2 \wedge dy.
$$
The condition given in the theorem is then the non-integrability assumption required.
\end{proof}

Now we give a simple example:
\begin{ex}
To create a non-integrable and non-differentiable example we have to satisfy $ (c_x- b_z+ b_yc - c_yb)(p)>0$ at some point where $\eta$ is non-differentiable. Consider
$$
b(x,y,z) = \text{sin}(y)e^{x^{\frac{1}{2}}}z \quad \quad
c(x,y,z) = \text{cos}(y)e^{(z+2)^{\frac{2}{3}}}x,
$$
which gives
\[
\begin{aligned}
  (c_x- b_z+ b_yc - c_yb)&= \\
\text{cos}(y)e^{(z+2)^{\frac{2}{3}}}- \text{sin}(y)e^{x^{\frac{1}{2}}}  &+ e^{(z+2)^{\frac{2}{3}}} e^{x^{\frac{1}{2}}}zx.\\
\end{aligned}
\]
Then for instance, $\eta$ (or any of its product with a differentiable function) is non-differentiable at $x=0, z=0, y=0$ but $\eta \wedge d\eta(0) = e^{2^{\frac{2}{3}}} dx \wedge dy \wedge dy$. Therefore there exists neighbourhoods on which $\eta$ is non-differentiable and yet satisfies the sub-Riemannian properties mentioned in this paper.

To create a non-H\"older example out of this, we can replace for instance $e^{x^{\frac{1}{2}}} $ with the function $f(x) = \frac{1}{\text{log}(x)}$ (setting it to be $0$ at $0$) and we have an example that is non-H\"older at $x=0, z=0, y=0$ and also non-integrable.
\end{ex}

To compare our results to the results given in \cite{MonMor12} and \cite{Kar11} we need to work with a certain basis of $\Delta$. The most canonical one is the ones that we have been working with, the adapted basis. In the case when $a=1$, the bundle is spanned by two vector-fields of the form $X_1 = \frac{\partial}{\partial x} + b \frac{\partial}{\partial y}, X_2 = \frac{\partial}{\partial z} + c \frac{\partial}{\partial y}$ and the differentiability assumptions
above mean that $[X_1,X_2]$ exists and is continuous but not Lipschitz continuous. If $a$ is non-constant then we have $X_1 = \frac{\partial}{\partial x} + \frac{b}{a}\frac{\partial}{\partial y}$, $X_2 = \frac{\partial}{\partial z} + \frac{c}{a}\frac{\partial}{\partial y}$  and it is not a priori clear that whether if the Lie derivatives can be defined since $a$ is not differentiable in $y$. However we see that we can define it via:
$$
[X_1,X_2]= d\eta(X_2,X_1)\frac{1}{a}\partial_y = -(c_x + b_z - \frac{ca_x + bc_y - ba_z +cb_y}{a}) \frac{\partial}{\partial_y}.
$$
One can check that this is the same expression as one would obtain in the $C^1$ case. The derivatives of $a$ with respect to $y$ disappear due to the symmetry in the vector-fields. And so if one mollifies $X_i$ to $X^{\epsilon}_i$, $[X^{\epsilon}_1,X^{\epsilon}_2]$ converges to $[X_1,X_2|$ above since also $\eta^{\epsilon}\rightarrow \eta$ and $d\eta^{\epsilon}\rightarrow d\eta$. Therefore we obtain continuous vector-fields which have continuous exterior differentialss.  This example is perhaps related to the question ''Does there exists Carnot manifolds such that $\Delta$ is $C^1$, and commutators of its vector-fields are lienar combinations of $C^1$-smooth basis vectors with continuous coefficients?'' posed in \cite{Kar11}. In this example the smooth basis vector is ${\partial_y}$ with the continuous coefficients as described above. The way we choose to define the Lie brackets above is due to the fact that because of the form of $X_i$ any ''loop'' that we construct using their integral curve ends up in the same $y$ axis as the starting point and therefore the Lie derivatives end up being tangent to $\partial_y$ direction. It is also not clear how one would go about defining the Lie derivatives of arbitrary basis vectors-fields.

It is also interesting to address the converse, that when can we construct a continuous exterior differential out of continuous Lie brackets. For this we need to give a definition for two continuous vector-fields to have a continuous Lie derivative, without using continuous exterior differentials. There may be several definitions the more geometric one being via loops. We will use another condition though. Let $\eta = a_0(x,y)dy + \dots$ and $\Delta = ker(\eta)$ as before with $a_0>0$. Assume we have a sequence of approximations $ker(\eta^k)=\Delta^k$ with any sequence of basis of sections $\{Z_i\}_{i=1}^n$ that converges to a basis of sections $\{Z_i\}_{i=1}^n$ of $\Delta$ and $\eta^k = a^k_0(x,y)dy + \dots$ such that $a^k_0>0$. Then, the regularity conditions we impose are the following:

\begin{itemize}
\item Locally the functions  $\eta^k([Z^k_{\ell},Z^k_{j}])$ converge uniformly to some functions (which can be seen as a weaker form of requiring existence of continuous Lie derivatives),

\item Locally the functions $Z^k_i(a^k_0)$ converge uniformly to some functions (which can be seen as a certain regularity assumption for the bundle $\Delta$ along its basis vector-fields),

\item Locally the functions $\eta([Z^k_i,\partial_y])$ converges uniformly to some functions (this will imply that $\Delta$ is differentiable along the $\partial_y$ direction).
\end{itemize}
Then, the functions
$$
d\eta^k(Z^k_i,Z^k_j) = \eta^k([Z^k_j,Z^k_i]),
$$
$$
d\eta^k(\partial_y,Z^k_i) = \eta^k([Z^k_i,\partial_y]) - Z^k_i(a^k_0),
$$
all converge and using bilinearity over smooth functions, we can show that for any 2 vector-fields $Z,Y$, $d\eta^k(Z,Y)$ converges which implies that $d\eta^k$ converges to a differential 2-form which is the requirement for the existence of continuous exterior differential. The conditions above can be stated in a more cordinate free way as follows: There exists a smooth vector-field $Y$, transverse to $\Delta$ such that
all the derivatives $[Z_i,Z_j]$, $[Z_i,Y]$ and $Z_i(\eta(Y))$ exist (using the definition above via some approximations). Then, the 2-form $d\eta$ is defined in a similar way. So this demonstrates that with some reasonable definition of a continuous Lie derivative for continuous vector-fields and some regularity assumptions one can get existence of continuous exterior differential more geometrically. However the existence of continuous Lie derivatives alone might not be enough. Note that the last two items are required in order to define $d\eta^k(\partial_y,Z^k_i)$. But actually there is a more geometric way to obtain control over the transversal behavior of $d\eta^k$ via contact structures. Let us elaborate. Assume we have a sequence of contact structures $\Delta^k$ which approximate $\Delta$. Then, the Reeb vector-fields of the approximations can help us control the transversal behavior of $d\eta^k$. The main idea is that if $R^k$ are the Reeb vector-fields of $\Delta^k = ker(\eta^k)$, then $d\eta^k(R^k,\cdot)=0$ to start with and all the regularity assumptions above about $Z^k_i(a^k_0)$ and  $\eta([Z^k_i,\partial_y])$ are not required. For certain other reasons, one requires $R^k$ not to converge to $\Delta$ in the limit however, which can be seen as a sort of non-involutivity condition that fits very naturally to our setting. Therefore we then only require convergence of  $\eta^k([Z^k_{\ell},Z^k_{j}])$. Hence in this setting existence of continuous Lie brackets defined as above seems to be equivalent to existence of continuous exterior differential. At this moment it is useful to note that if we have a continuous exterior differential and non-integrability, then we automatically obtain a local sequence of contact structures that approximate our bundle. Therefore it becomes meaningful to ask whether if the existence of continuous exterior differential can be completely replaced by the existence of a sequence of approximating contact structures. Given the natural relation between contact structures and step 2 completely non-integrable bundles, it seems like a worth while direction to explore.

So coming back to comparisons, for the particular examples that one can construct with Proposition \ref{prop-example}, the results of \cite{MonMor12} can not be applied since they require $X_1,X_2$ and $[X_1,X_2]$ to be Lipschitz continuous. Lipschitzness for instance, grants the critical property of existence of unique solutions for these vector-fields which is an essential tool in Lipschitz continuous analysis and which we have obtained through more geometric means thanks to Stokes property (for $X_i$). However we can not at the moment also say that our results cover all the step $2$ systems covered by their conditions. The conditions they have stated in their paper probably does not necessitate the existence of an exterior differential as described above. We would only have that the sequences of functions $\eta^k([X^k_{\ell},X^k_{j}])$,  $X^k_i(a^k_0)$,  $\eta^k([X^k_i,\partial_y])$ obtained via mollifications have uniformly bounded norms with respect to $k$ and converge a.e. We do believe that there is hope to extend the theory present in this paper to that direction though, and this is discussed in the final section.

Considering the paper \cite{Kar11}, the step 2 case of the Ball-Box result given there, although can be covered by our theorem, is already covered by results given in \cite{Gro96}. Our main motivation here is to in fact work with non-differentiable bundles (due to our interest in non-differentiable bundles that arise in dynamical systems) so the results of \cite{Kar11} do not directly apply to the class of examples we are interested in. However the fact that the authors there can simply work with continuous Lie derivatives and still get the results for all steps is already quite remarkable and can be seen as an extension of results of \cite{Gro96} to arbitrary step cases. Also the existence of continuous Lie derivatives is definitely related to existence of continuous exterior differentials as discussed above. Therefore, in some sense, our result can be seen as a step 2 version of some of the results obtained in \cite{Kar11} in which the differentiability assumption for the bundle is removed and the continuous Lie derivatives assumption is replaced with continuous exterior differential assumption.

Now we are back to studying further properties of continuous exterior differentiability. In particular we will build examples of such bundles on manifolds and not just local neighborhoods.
We now give a lemma from \cite{Har64} that is helpful in generating new examples of elements of $\Omega^1_d(M) \setminus \Omega^1_1(M)$ from given ones.

\begin{lem}Let $\eta$ be an element of $\Omega^1_d(M)$. Then, given any $C^1$ function $\phi$ one has that
$\phi\eta$ is an element of $\Omega^1_d(M)$ with continuous exterior differential $d\phi \wedge \eta + \phi d\eta$. \end{lem}

Of course if $\phi$ is nowhere $0$ from the point of view of integrability this construction does not change anything since
$\eta \wedge d\eta >0$ implies $\phi \eta \wedge d(\phi \eta) >0$ and similarly for being equal to $0$. The importance of this lemma however lies in the fact that it allows us to paste together local elements of $\Omega^1_d(M)$ (which were shown to exist above). We now explain this. Let $\{U_i\}_{i=1}^{\infty}$ be a collection of local coordinate neighbourhoods that cover $M$. Assume they are equipped with local differential forms $\alpha_i$ defined on $U_i$ which are elements of $\Omega^1_d(V_i) \setminus \Omega^1_1(V_i)$ for some $V_i  \subset U_i$ and a partition of unity $\{\psi_i\}_{i=1}^{\infty}$ such that $\psi_i|_{V_i}=1$. As a direct corollary of previous lemma (and the finiteness of overlapping partition of unity cover elements) we obtain

\begin{lem} The 1-form defined by $\eta = \sum_{i=1}^{\infty} \psi_i \eta_i$ is an element of $\Omega^1_d(M)$ $\setminus \Omega^1_1(M)$ with continuous exterior differential $d\eta = \sum_{i=1}^{\infty} d\psi_i \wedge \eta_i + \psi_i d\eta_i$.
\end{lem}

Of course even if every $\alpha_i$ is everywhere non-integrable on each $U_i$ we can only guarantee that $\eta$ would be non-integrable at certain points and not everywhere on $M$. This is similar to not being able to paste together local contact structures to form a global one (in general). It would be indeed very interesting to have an example of an element of $\Omega^1_d(M) \setminus \Omega^1_1(M)$, for some $M$, which is everywhere non-integrable. It would then also make more sense to generalize fundamental theorems of contact geometry to this setting.
A good place to start would be Anosov flows as it is known that the continuous center-stable and center-unstable bundles of Anosov flows can be approximated by smooth contact structures \cite{Mit95}.

We now prove an analytic property. For this we define the function $|\cdot|_d: \Omega^k_d(M) \rightarrow \mathbb{R}$, $|\beta|_d = \max\{|\beta|_{\infty}, |d\beta|_{\infty}\}$ where we assume that $M$ is compact (or if not that we work with only compactly supported differential forms).

\begin{lem}The space $\Omega^n_d(M)$ equipped with $|\cdot|_d$ is a Banach space over $\mathbb{R}$. \end{lem}
\begin{proof} It is easy to establish that $|\cdot|_d$ is a norm. Now assume $\beta^k \in \Omega^n_d(M)$ is a Cauchy sequence with respect to the given norm. This means that $\beta^k$ and $d\beta^k$ are Cauchy sequences with respect to supnorm over $M$. This means that $\beta^k$ converges uniformly to some $n-$form $\beta$ and $d\beta^k$ converges uniformly to some $n+1-$form $\alpha$.
 Then for any $n+1$ chain $S$ bounded by some $n$ chain $c$, we have by uniform convergence
$$
\int_c \beta = \lim_{k \rightarrow \infty} \int_c \beta^k = \lim_{k \rightarrow \infty} \int_S d\beta^k =   \int_S \alpha.
$$
This means that $\alpha$ is the continuous exterior differential of $\beta$
so $\beta \in (\Omega^n_d(M), |\cdot|_{d})$.
\end{proof}

Finally we prove an algebraic property.
\begin{lem}\label{lem-dd=0} $d$ maps $\Omega^n_d(M)$ to $\Omega^{n+1}_d(M)$ such that $d^2=0$. In particular the sequence $0 \rightarrow \dots \Omega^n_d(M) \rightarrow \Omega^{dim(M)}_d(M) \rightarrow 0$ is exact. \end{lem}
\begin{proof} Let $\eta$ be in $\Omega^n_d(M)$ and $d\eta$ be its continuous exterior differential.
Then there exists a sequence of $C^1$ $k+1$ forms $d\eta^k$ that converges to $d\eta$. Moreover $dd\eta^k=0$. Therefore by Proposition \ref{prop-char} $d\eta$ has continuous exterior differential $dd\eta=0$. So $d\eta \in \Omega^{n+1}_d(M)$ and $dd\eta=0$.  \end{proof}

\subsection{Integrability of Bunched Partially Hyperbolic Systems}
In this subsection we give a tentative application to dynamical systems. It is tentative because although we state a novel integrability theorem for a class of bundles that arise in dynamical systems, we can not yet construct any examples that satisfy the properties. However we decided to include it in this paper first of all because it conveys the potential of continuous sub-Riemannian geometry for applications and secondly if the generalizations stated in Section \ref{sec-generalizations} can be carried out then it will most likely be possible to improve this theorem and construct examples of dynamical systems that satisfy it.

Let $M$ be a compact Riemannian manifold of dimension $n+1$ and $f: M \rightarrow M$ is a diffeomorphism.
Assume moreover that there exists a continuous splitting $T_xM = E^s_x \oplus E^c_x \oplus E^u_x$, each of which is invariant under $Df_x$. This splitting is called partially hyperbolic if there exists constants $K, \lambda_{\sigma}, \mu_{\sigma}>0$ for $\sigma = s,c,u$
such that $\mu_{s} < \lambda_{c}$, $\mu_{s} \leq 1$, $\mu_c < \lambda_u$, $\lambda_u>1$ and for all $\sigma = s,c,u$, $\ i>0 \ (i \in \mathbb{Z}^+),$ $x\in M$ and $v_{\sigma} \in T_xM$ such that $|v_{\sigma}|=1$
$$
 \frac{1}{K}\mu^i_{s} \leq |Df^i v_{\sigma}| \leq K\lambda^i_{\sigma}.
$$
This basically means that under iteration of $f$, the $Df$ expands $E^u$ exponentially and contracts $E^s$ exponentially while the behavior of $E^c$ is in between the two. Although these bundles might be just H\"older continuous (see \cite{HasWil99}), a well known property of such a system is that $E^u$ and $E^s$ are uniquely integrable into what are called as unstable and stable manifolds. Given $p \in U$, we denote the connected component of the stable and unstable manifold in $U$ that contains $p$ as $\mathcal{W}^u_p$ and $\mathcal{W}^s_p$ and call them local stable and unstable manifolds. In general
$E^{c}$, $E^{cs}=E^{c} \oplus E^s$ or  $E^{cu}=E^{c} \oplus E^u$ maybe be non-integrable both in the case the bundles are continuous (see \cite{HerHerUre15}) and or differentiable (see \cite{BurWil07}). The integrability of these bundles play an important role in classification of the dynamics, see for instance \cite{HamPot13}. It is our aim to apply now Theorem \ref{thm-ballbox} to get a novel criterion for integrability of these bundles under additional assumptions on geometry and dynamics.

A dynamical assumption that we will make is center bunching. Conditions like center bunching appears quite commonly in studies of dynamical systems. See for instance \cite{BurWil10} where it plays an important role for ergodicity.  A system is called center-bunched if $\frac{\lambda^2_c}{\mu_u}<1$. It means that the expansion in the unstable direction strongly dominates the expansion in center (as opposed to the definition of partially hyperbolic system where one only has $\frac{\lambda_c}{\mu_u}<1$). In \cite{BurWil07} Theorem 4.1, the authors prove that if $E^c$ and $E^s$ are $C^1$ and center bunched, partially hyperbolic then $E^{cs}$ is uniquely integrable. s far as we are aware there is no general result on integrability of such continuous bundles that do not make any assumptions on the differentiable and topological properties of the manifold $M$ (see for instance \cite{BriBurIva09} where they assume $M$ is a torus or \cite{HamPot13} where they have assumptions on the fundamental group of the manifold). An integrability theorem for continuous bundles that only make assumptions on the constants above would indeed be quite useful. Since being a partially hyperbolic system and center bunching are preserved under $C^1$ perturbations of $f$ they constitute an open set of examples (in $C^1$ topology) inside partially hyperbolic systems.

Our aim is to make one small step towards an integrability condition for continuous bundles that relies only on the constants. The only place where differentiability is required in the proof of the theorem in \cite{BurWil07} is where certain sub-Riemannian properties (more specifically the smaller box inclusion in the Ball-Box Theorem) are required. Thus once these properties are guaranteed then the proof easily carries through.

\begin{thm} Assume $f: M \rightarrow M$ is a diffeomorphism of a compact manifold which admits a center bunched partially hyperbolic splitting $T_xM = E^s_x \oplus E^c_x \oplus E^u_x$ where $dim(E^u_x)=1$. Assume moreover that $\mathcal{A}^1_0(E^{cs})$ admits a continuous exterior differential. Then $E^{cs}$ is uniquely integrable with a $C^1$ foliation.
\end{thm}
\begin{proof} As in \cite{BurWil07}, one starts by assuming that there exists a point $p$ where $\eta \wedge d\eta(p)>0$ where
$E^{cs}= ker(\eta)$ and $d\eta$ is the continuous exterior differential. Then by Theorem \ref{thm-ballbox}, there exists a $C^1$ adapted coordinate system and a neighbourhood $U$ of $p$ on which every point $q$ on the local unstable manifold $\mathcal{W}^u_p$ of $p$ can be connected to $p$ by a length parametrized admissible path $\kappa(t)$ such that $\kappa(0)=p$, $\kappa(\ell(\kappa))=q$ and $d(p,q) \geq c|\kappa|^2$. To show that this can be done,
we choose first a smooth adapted coordinate system at $p$ for $E^{sc}$ so that the unstable bundle is very close to the $\partial_y$ direction (since both are transverse to $E^{sc}$ this is possible). By choosing it close enough, we can make sure that when we pass to the $C^1$ adapted coordinates using the transformation given in equation \eqref{eq-c1adapted}, the unstable direction and $\partial_y$ direction are still close enough so that in a small enough neighborhood $U$ and for any $q=(x,y) \in \mathcal{W}^u_p$,  $|x^i| \leq \delta|y|$ where $\delta<\frac{1}{2n}$. Then to apply the Ball-Box Theorem in this $C^1$ adapted coordinate system,
for $\epsilon$ small enough, we pick $q=(x,y) \in \mathcal{B}_2(0,K_1,\epsilon)$ so that $|y| = K_1\epsilon^2$. But the Ball-Box Theorem tells us that there exists a length parametrized admissible curve $\kappa$ such that $\ell(\kappa) \leq \epsilon$, $\kappa(0)=p$ and $\kappa(\ell(\kappa))=q$. Since
$|x^i| \leq \frac{K_1}{2n}|y|$, we have $d(q,p) \geq \frac{1}{2}|y| = \frac{K_1}{2}{\epsilon^2}$ (where in this coordinate system we remind that $p=0$). Therefore for some constant $c$, $d(q,p) \geq c|\kappa|^2$.
Thus the conditions 1 to 4 appearing in the proof of Theorem 4.1 of \cite{BurWil07} are fully satisfied and the rest of the analysis only depends on the dynamics of $f$. So by the same contradiction obtained there we get that $\eta \wedge d\eta=0$ everywhere. Then by the integrability theorem of Hartman in \cite{Har64}, this means that $E^{cs}$ integrates to a unique $C^1$ foliation. \end{proof}

\section{Some Perspectives Regarding Generalizations} \label{sec-generalizations}

In this section we ask some questions that are related to generalizations of the theorems stated in this paper.

\subsection{Relaxing Existence of Continuous Exterior Differential: Exterior Regularity}

One meaningful way to relax the condition on existence of a continuous exterior differential of $\eta$ is to require the following
\begin{itemize}
  \item There exists a sequence of $C^1$ differential forms $\eta^k$ which converge uniformly to $\eta$
        such that $|d\eta^k|_{\infty} \leq C$ for all $k$.
\end{itemize}

Then the non-integrability at $p$ condition could be stated as
\begin{itemize}
  \item There exists a constant $c>0$ and a neighbourhood $U$ of $p$ such that for all $k$ $(\eta^k \wedge d\eta^k)_q>c$ for all $q \in U$.
\end{itemize}

In this case we will be working with a sequence of differential forms $\eta^k$ and $d\eta^k$ acting on objects from $\Gamma(\Delta)$ all of which is defined on an adapted coordinate system with respect to $\Delta$ on some neighbourhood $U$. Once $U$ is fixed, one requires that its size does not change with respect to $\eta^k$. That is we should be able to satisfy the conditions given in subsection \ref{sssection-fixU} on a fixed $U$ for all $\eta^k$. This is the first reason why we require non-integrality on a fixed neighbourhood $U$ of $p$ since otherwise $\eta^k$ could be non-integrable on smaller and smaller domains whose size shrink to $0$ forcing us to decrease the size of $U$ as well. We also see that the condition $|d\eta^k|_{\infty} \leq C$ is important in being able to satisfy the other requirements given in subsection \ref{sssection-fixU} with respect to all $\eta^k$ on a fixed domain $U$.

It has already been shown in previous work \cite{LuzTurWar16} that under this condition one of the crucial lemmas which is \ref{prop-uniquebasis} holds true. Moreover one has that for every k-cycle $Y$ and k+1 chain $H$ bounded by it $\int_Y \eta = \lim_{k\rightarrow \infty}\int_H d\eta^k$ and also since $\eta^k$ converges to $\eta$ and $|d\eta^k|_{\infty}$ is uniformly bounded then $(\eta \wedge d\eta^k)_q$ can be made arbitrarily close to $(\eta^k \wedge d\eta^k)_q$ by taking $k$ large enough and hence non-zero. So one simply replaces $d\eta(X_i,X_j)$ with $d\eta^k(X_i,X_j)$. Although $\eta^k$ does not annihilate curves tangent to $X_i$, since it converges to $\eta$, this difference can be made arbitrarily small by taking $k$ large enough and the analysis will carry through.

This generalization, if done, may allow one to replace the example \ref{prop-example} which was $\eta = a(x,y,z)dy + b(x,y,z)dx + c(x,y,z)dz$, by a more general one where $b$ is only Lipschitz in $y$ and $z$, $c$ is only Lipschitz in $y$ and $x$ and $a$ is only Lipschitz in $x$ and $z$ (instead of the $C^1$ assumption). But note that the non-integrability condition will be more tricky to check.

\subsection{Higher Coranks}

Assume now that $\Delta$ is a corank $m$ tangent subbundle in a $m+n$ dimensional manifold. As pointed out after equation \eqref{eq-adapted}, we can still find an adapted basis $X_i$ for such a bundle. We also assume that on some $U$, $\mathcal{A}^1_0(\Delta)$ is spanned by $\{\eta_i\}_{i=1}^m$ with exterior differential $\{d\eta^i\}_{i=1}^m$. The case where
\begin{equation}\label{eq-nonint}
(\eta^1 \wedge \dots  \eta^m \wedge d\eta^{\ell})(p)>0,
\end{equation}
for all $\ell = 1,\dots ,m$ represents the case of higher co-rank but still step 2 completely non-integrable case. In this case the transversal direction will not be one dimensional so the proof may become conceptually harder to carry out. But it seems to the authors that the main change appearing will only be the replacement of terms $|\eta(\partial_y)|_{\infty}$ and $|\eta(\partial_y)|_{\inf}$ with $|\eta_1\wedge \dots \wedge \eta_n|_{\infty}$ and $m(\eta_1\wedge \dots \wedge \eta_n)_{\inf}$.

Higher step cases might be impossible to carry out in full generality however we believe that there might be a subclass on which this approach may be generalized. Assume $\Delta$ is a corank 1 bundle so that $\mathcal{A}_1^0(\Delta)$ is equipped with a continuous exterior differential $\{V,\eta_i, d\eta_i\}$ for $i=1,..,m$. Let $\{Y_i\}_{i=1}^m$ be a set of vector-fields that are inside $\text{span}\{\frac{\partial}{\partial y^i}\}_{i=1}^m$ and such that $\eta_j(Y_i)=\delta_{ij}$. Let also $\{X_i\}_{i=1}^n$ be the usual adapted basis. Then we can define the Lie bracket of $X_i$ via
\begin{equation}\label{eq-continousliebracket}
[X_i,X_j] = -\sum_{i=1}^m d\eta_i(X_j,Z_i)Y_i.
\end{equation}
Note that again by the form the adapted basis, any loop constructed from any pair of such vector-fields always stays in the $y$ plane. Now define $\Delta_0=\Delta$, $\Delta_1 = \Delta_{0} + [\Delta_0,\Delta_0]$. Assume $\mathcal{A}_1^0(\Delta_1)$ is equipped with an exterior differential $\{V,\eta_i, d\eta_i\}$ for $i=1,\dots,m-k_1$ for $1 \leq k_1 \leq m$. Then one can also define $[\Delta_1,\Delta_1]$ as above and then define  $\Delta_2 = \Delta_{1} + [\Delta_1,\Delta_1]$. Proceeding inductively with always the assumption of existence of continuous exterior differentials and the strict inclusion $\Delta_{i+1} \varsupsetneq \Delta_i$ (since $k_i \neq 0$) we obtain a chain
$\Delta_0 \varsubsetneq \Delta_1 \varsubsetneq \dots \varsubsetneq \Delta_{\ell}$ which terminates for some $\ell$ such that $\Delta_{\ell} = TM$. The meaningful question to ask then is whether if analogues of The Ball-Box and Chow-Rashevskii Theorem hold true in this case. This is very much akin to the requirements in \cite{MonMor12}, where for higher step cases, one requires higher order Lie brackets to be Lipschitz continuous. Of course finding an example of a bundle that satisfies the properties above will be substantially harder, so one might first try to find such an example based on the examples given in this paper before embarking on trying to prove the generalization.

\subsection{More Generally, H\"older Continuous Bundles} \label{subsec-holder}

Now we explain a fundamentally more difficult generalization which the authors think is true but are not able to prove yet. We want to remove both the existence of $d\eta$ and boundedness of $|d\eta^k|$ explained in the previous section all together so that the applicability range of this theorem increases greatly. Namely, assuming corank 1, we only want to impose the following: There exists a sequence of differential 1-forms $\eta^k$ that converge to $\eta$ uniformly and for some neighbourhood $U$ of $p$,
$(\eta^k \wedge d\eta^k)_q > 0$ for all $q \in U$. Note that we still have one fundamental equality satisfied:
For every k-cycle $Y$ and k+1 chain $H$ bounded by it $\int_Y \eta = \lim_{k\rightarrow \infty}\int_H d\eta^k$. This is of course just one important step of the analysis. We lose one crucial property, we lose the fact that the adapted basis $\{X_i\}$ are uniquely integrable. This brings about the problem of choosing certain integral curves to build something similar to the surface $\mathcal{W}_{\epsilon}$ that was used in the construction of the accessible neighbourhood. At this part, in corank $1$ the notion of maximal and minimal solutions can be of help to determine in a well defined way some objects similar to $\mathcal{W}_{\epsilon}$. The problems don't end here however. Note that in application of Stokes property with $d\eta^k$ we will need an estimate on objects like $d\eta^k(X_i,X_j)$. The fact that $|d\eta^k|$ might not be bounded may cause problems in conditions required in subsection \ref{sssection-fixU}. However it seems likely that with some restrictive relations between how fast $\eta^k$ converges and
how fast $|d\eta^k|$ may blow up these conditions can still be satisfied in certain cases. At this stage using for instance mollifications as the approximation could be useful as one can write down the relation between such terms more precisely, as was done in \cite{LuzTurWar16}. Although the authors are hopeful about this generalization, they are not completely sure whether if the analysis carries through or not. It will be subject of future works.

\medskip
S\.ina T\"urel\.i \\
\textsc{Imperial College, South Kensington, London}\\
 \textit{Email address:} \texttt{sinatureli@gmail.com}

\medskip Sina T\"ureli \\ \textsc{Imperial College London, South Kensington Campus, London} \\
 \textit{Email address:} \texttt{sinatureli@gmail.com}


\begin{thebibliography}{1}


\bibitem{Agr12}
\textit{A.Agrachev, D. Barilari, U. Boscain}
\newblock{\em Introduction to Riemannian and sub-Riemannian Geometry}
\newblock{Preprint, SISSA, (2012)}

\bibitem{Agr04}
\textit{A.Agrachev, Yu.L. Sachkov}
\newblock{\em Control Theory from the Geometric Viewpoint}
\newblock{Encyclopaedia of Mathematical Sciences (Book 87),Springer; (2004) edition}

\bibitem{Arn89}
\textit{V.I Arnol'd}
\newblock{\em Mathematical Methods of Classical Mechanics}
\newblock{Springer-Verlag New York, 1989, Second Edition}

\bibitem{Bel96}
\textit{A. Bella\"iche}
\newblock{\em The Tangent Space in sub-Riemannian Geometry}.
\newblock {In sub-Riemannian geometry, volume 144 of \emph{Progr. Math.}, pages 1-78. Birkh\"auser, Basel, 1996.}

\bibitem{BriBurIva09}
\textit{M. Brin, D. Burago, and S. Ivanov.}
\newblock {\em Dynamical coherence of partially hyperbolic diffeomorphisms of the 3-torus.}
\newblock{ Journal of Modern Dynamics, 3(1):1–11, 2009.}

\bibitem{BurWil07}
\textit{K. Burns, A. Wilkinson}
\newblock {\em Dynamical coherence and center bunching}
\newblock {Discrete and Continuous Dynamical Systems, (Pesin birthday issue) 22 (2008), 89-100.}

\bibitem{BurWil10}
\textit{K. Burns, A. Wilkinson}
\newblock {\em On the ergodicity of partially hyperbolic systems}
\newblock {Annals of Mathematics,  Pages 451-489, Volume 171 (2010).}

\bibitem{Car09}
\textit{C. Carath\'eodory}
\newblock{\em Researches on the foundations of thermodynamics}
\newblock{Math. Ann. 67 (1909) 355–386}


\bibitem{Gro83}
\textit{M. Gromov}
\newblock{\em Filling Riemannian manifolds}
\newblock{J. Diff. Geometry 18 (1983), 1–147}

\bibitem{Gro96}
\textit{M. Gromov}
\newblock {\em Carnot-Carathéodory spaces seen from within}
\newblock {In sub-Riemannian geometry, vol. 144 of \emph{Progr. Math.} , pages 79-323, Birkh\"auser, (1996),}


\bibitem{HamPot13}
\textit{R. Potrie, A. Hammerlindl}
\newblock{\em Classification of partially hyperbolic diffeomorphisms in 3-manifolds with solvable fundamental group}
\newblock{Submitted}


\bibitem{Har64}
\textit{P. Hartman},
\newblock{\em Ordinary Differential Equations.}
\newblock{Society for Industrial and Applied Mathematics. (1964)}


\bibitem{Har66}
\textit{P. Hartman},
\newblock{\em Frobenius Theorem under Carathéodory Type Conditions}
\newblock{J. Dif. Eq. 7, 307-333 (1970)  }

\bibitem{HasWil99}
\textit{B. Hasselblatt, A. Wilkinson}
\newblock{Prevalance of non-Lipschitz Anosov Foliations.}
\newblock{Erg. Th. Dyn. Syst. 19 (1999), no. 3, pp. 643-656. }

\bibitem{HerHerUre15}
\textit{F. R. Hertz, J.R. Herz, R. Ures}
\textit{\em A non-dynamically coherent example on T3}
\textit{Annales de l'Institut Henri Poincare (C) Non Linear Analysis.}


\bibitem{Kar11}
\textit {M. Karmanova}
\newblock{\em A new approach to investigation of Carnot-Carath\'eodory geometry.}
\newblock{ Geom. Funct. Anal., 21(6):1358–1374, 2011}


\bibitem{LuzTurWar16}
\textit{S. Luzzatto, S. T\"ureli, K.M.War}
\newblock{\em Integrability of Continuous Bundles}
\newblock{to be published in J. reine angew. Math, arXiv:1606.00343 (2016).}

\bibitem{Mit95}
\textit{Y. Mitsumatsu}
\newblock{\em Anosov flows and non-stein symplectic manifolds, Ann.}
\newblock{Inst. Fourier Grenoble 45(5) (1995) 1407–1421.}

\bibitem{Mon02}
\textit{R. Montgomery}
\newblock{\em A tour of sub-Riemannian geometries, their geodesics and applications}
\newblock{Mathematical Surveys and Monographs, vol. 91, American Mathematical Society, Providence,RI,
(2002)}


\bibitem{MonMor12}
\textit {A. Montanari and D. Morbidelli}
\newblock{\em Nonsmooth H\"ormander vector fields and their control balls.}
\newblock{Trans. Amer. Math. Soc., 364(5):2339–2375, 2012.}



\bibitem{Li10}
\textit{L. Shin}
\newblock{\em Control of a Network of Spiking Neurons}
\newblock{IFAC Proceedings Volume 43, Issue 14, (2010), Pages 314–319}

\bibitem{Sim10}
\textit{S. N. Simi\'c}
\newblock{\em A Lower Bound for the subriemannian Distance of H\"older Continuous Distributions}
\newblock{Proc. Amer. Math. Soc. 138 (2010), 3293-3299 }


\end{thebibliography}
\end{document}